\documentclass[journal,onecolumn,web]{ieeecolor}
\usepackage{generic}
\usepackage{cite}
\usepackage{amsmath,amssymb,amsfonts}
\usepackage{algorithmic}
\usepackage{graphicx}
\usepackage{algorithm,algorithmic}
\usepackage{hyperref}
\usepackage{amssymb}
\usepackage{verbatim}
\usepackage{mathrsfs}

\usepackage{booktabs}

  \makeatletter
  \let\NAT@parse\undefined
  \makeatother

\usepackage{tikz}
\usetikzlibrary{arrows, shapes, chains}
\tikzstyle{startstop} = [rectangle,rounded corners, minimum width=3cm,minimum height=1cm,text centered, draw=black,fill=red!30]
\tikzstyle{io} = [trapezium, trapezium left angle = 70,trapezium right angle=110,minimum width=3cm,minimum height=1cm,text centered,draw=black,fill=blue!30]
\tikzstyle{process} = [rectangle,minimum width=5cm,minimum height=2cm,text centered, text width =5cm,draw=black,fill=white]
\tikzstyle{decision} = [diamond,minimum width=3cm,minimum height=1cm,text centered,draw=black,fill=green!30]
\tikzstyle{arrow} = [thick,->,>=stealth]
\usepackage{subcaption}
\usepackage{amsthm}
  \theoremstyle{plain}
  \newtheorem{theorem}{Theorem}
  \newtheorem{lemma}[theorem]{Lemma}
  \newtheorem{proposition}[theorem]{Proposition}
  \newtheorem{corollary}[theorem]{Corollary}
  \newtheorem{assumption}{Assumption}
  \newtheorem{definition}{Definition}
  \newtheorem{example}{Example}
  
  \theoremstyle{remark}
  \newtheorem{remark}{Remark}
\hypersetup{hidelinks=true}

\usepackage{textcomp}
\def\BibTeX{{\rm B\kern-.05em{\sc i\kern-.025em b}\kern-.08em
    T\kern-.1667em\lower.7ex\hbox{E}\kern-.125emX}}
\allowdisplaybreaks[4]
\begin{document}
\title{{An Optimization-Based Framework for Solving Forward-Backward Stochastic Differential Equations: Convergence Analysis and Error Bounds}}
\author{Yutian Wang, Yuan-Hua Ni, Xun Li,
\thanks{Yutian Wang is with the Department of Applied Mathematics, The Hong Kong Polytechnic University, China. Email: {\tt yutian.wang@connect.\linebreak polyu.hk}.
Yuan-Hua Ni is with the College of Artificial Intelligence, Nankai University, Tianjin, China. Email: {\tt yhni@nankai.edu.cn}.
Xun Li is with the Department of Applied Mathematics, The Hong Kong Polytechnic University, China. Email: {\tt li.xun@polyu.edu.hk}.
}
}

\maketitle

\begin{abstract}
Forward-backward stochastic differential equations have recently
become a key focus in the computational field, and their role in
continuous-time stochastic optimal control and reinforcement learning
has grown increasingly prominent.
In this paper, we develop an optimization-based framework for solving
coupled forward-backward stochastic differential equations,
which naturally arise in stochastic optimal control through the
stochastic maximum principle and related Hamiltonian systems. We
introduce an integral-form objective function and prove its
equivalence to the error between consecutive Picard iterates. Our
convergence analysis establishes that minimizing this objective
generates sequences that converge to the true solution.
We provide explicit upper and lower bounds that relate the objective value to the error between trial and exact solutions.
We validate the proposed objective and its theoretical interpretation
using two analytical test cases,
and further illustrate its numerical applicability on a nonlinear stochastic optimal control problem with up to 1000 dimensions.
\end{abstract}


\section{Introduction}
\label{section:introduction}
Forward-backward stochastic differential equations (FBSDEs) are
coupled systems of stochastic differential equations that evolve both
forward and backward in time. Their significance in stochastic control
theory is well established: the stochastic maximum principle reveals
that the necessary conditions for optimality can be equivalently
formulated as an FBSDE system, which incorporates the state, adjoint
processes, and Hamiltonian dynamics
\cite{yongStochasticControlsHamiltonian1999}.
FBSDEs also underpin applications ranging from
mathematical finance, where they overcome limitations of the classical
Black--Scholes model by incorporating features such as default risk and
non-tradable assets. For reinforcement learning, they
can characterize value functions and policy dynamics
\cite{eMultilevelPicardNumerical2019,jiaPolicyGradientActorcritic2022,
jiaQlearningContinuousTime2023,hamblyRecentAdvancesReinforcement2023}.
This highlights the need for efficient and reliable
computational approaches, especially given the high dimensionality and
nonlinearity typical of practical problems.

Like partial differential equations (PDEs), analytically solving
FBSDEs is often intractable, and numerical methods are inevitable for
examining the solutions. Early studies of numerical methods emerged
soon after the general theory of nonlinear backward stochastic differential equations (BSDEs)
\cite{pardouxAdaptedSolutionBackward1990,pardouxBackwardStochasticDifferential1992}. From
a computational viewpoint, we can classify these numerical methods
into two categories: the PDE-based approach
\cite{maSolvingForwardbackwardStochastic1994,douglasNumericalMethodsForwardbackward1996}
and the conditional expectation approach
\cite{bouchardDiscretetimeApproximationMonteCarlo2004,zhangNumericalSchemeBSDEs2004,benderTimeDiscretizationMarkovian2008}. The
PDE-based approach relates FBSDEs to associated PDEs and relies on
numerical methods for PDEs to obtain the solution of FBSDEs. It is
worth noting that the converse application is also valid via nonlinear
Feynman-Kac formulae, where a numerical method for FBSDEs naturally
leads to an equivalent method for a certain class of PDEs. The
conditional expectation approach, on the other hand, works by
discretizing the time axis and expressing solutions recursively via
conditional expectations. With recent advances in deep learning, there
are also studies applying neural networks to solving PDEs and
computing conditional expectations
\cite{hanSolvingHighdimensionalPartial2018,beckMachineLearningApproximation2019,hureDeepBackwardSchemes2020}. For
an overview of existing numerical methods for FBSDEs, we refer to the
comprehensive survey by
\cite{chessariNumericalMethodsBackward2023}.

This work is closely related to several optimization-based BSDE methods.
The deep BSDE
method \cite{hanSolvingHighdimensionalPartial2018} solves
nonlinear BSDEs by treating the initial value \(Y_0\) of the backward process
and the control process \(Z\) as decision variables,
minimizing the terminal error. The martingale method
\cite{jiaPolicyEvaluationTemporaldifference2022} solves linear
BSDEs by treating the backward process \(Y\) as the decision variable,
minimizing an integral-form objective function. A notable property of
the objective function of martingale method is that it strictly equals
the mean squared error between the trial solution \(Y\) and the true
solution \(Y^*\). Motivated by this property,
\cite{wangDeepBSDEMLLearning2022} proposes a deep BSDE variant by
dropping \(Y_0\) and only optimizing the backward process \(Z\). Their
objective function is also proven equal to the mean squared error
between the trial solution \(Z\) and the true solution
\(Z^*\). \cite{anderssonConvergenceRobustDeep2023} further develop
this idea to coupled FBSDEs arising in stochastic optimal control problems.
These works show that FBSDEs can be solved effectively through optimization,
but they are introduced through problem-specific parameterizations and losses.

This work aims to develop a unified optimization-based framework
for solving FBSDEs without explicit time discretization.
By treating
both the backward process \(Y\) and the control process \(Z\) as
decision variables,
while letting the forward process $X$ be induced by the forward SDE,
we can obtain the \emph{true solution} \((Y^*,Z^*)\)
by minimizing an integral-form objective function over
\emph{trial solutions} \((Y,Z)\). This objective function,
called \emph{backward measurability loss} (BML),
is proven to equal the error between \((Y,Z)\) and
\(\Phi(Y,Z)\), where \(\Phi\) denotes the Picard operator for FBSDEs.
The BML concept originates from
\cite{wangProbabilisticFrameworkHowards2023}, initially defined
for linear BSDEs in policy evaluation.
Our framework thus
unifies the deep BSDE and the martingale method, extends them
to a general form of coupled FBSDEs, and recovers them under specific
trial solution designs. Certain deep BSDE method variants are likewise
encompassed.

Another major concern of this work is the convergence behavior during
optimization.
In existing optimization-based approaches, even if the
objective function is designed to vanish at the true solution, it is
generally unclear whether a trial solution with a small objective value
is actually close to the true solution, especially for coupled FBSDEs.
When the trial solution \(Y\) and/or \(Z\) is updated iteratively by
minimizing a loss function, the way in which the iterates approach the
true solution \((Y^*,Z^*)\) remains largely unclear. Existing
mean-squared error interpretations
\cite{jiaPolicyEvaluationTemporaldifference2022,wangDeepBSDEMLLearning2022}
apply only to certain linear BSDEs. Moreover, to the best of our
knowledge, among works on coupled FBSDEs, error-bound guarantees have
only been established in
\cite{benderTimeDiscretizationMarkovian2008,hanConvergenceDeepBSDE2020},
where the analysis is restricted to cases in which the forward equation
does not contain \(Z\). As a result, for more general coupled FBSDEs,
there is still limited theory explaining whether a small loss value
indeed implies that the trial solution is close to the true solution.

We leverage the Picard interpretation of our objective function to
analyze the convergence behavior, which is proven valid for
general coupled FBSDEs. By exploiting the Lipschitz continuity of the
Picard operator, we establish convergence results that guarantee that
minimizing the proposed objective function yields sequences converging
to the true solution. We also derive error bounds relating the
objective value of trial solutions to their distance from the true
solution. The theory is developed in continuous time, and time
discretization is introduced only for integral estimations. This
treatment simplifies analysis and yields a numerical method agnostic to
time discretization schemes. To our knowledge, these convergence
results and error bounds---natural consequences of the Picard
interpretation---are novel. Compared with analyses requiring weak
coupling conditions
\cite{benderTimeDiscretizationMarkovian2008,hanConvergenceDeepBSDE2020},
our convergence and error analysis is established under the same contraction-type assumptions used in the standard fixed-point well-posedness theory recalled in the appendix;
for the decoupled case, this reduce to the standing assumption for BSDEs.

\emph{Contributions.} Our main contributions are summarized as
follows. First, we propose the BML value to quantify how well a trial
solution satisfies an FBSDE, and prove it equal to the error between
consecutive points in Picard iteration. Second, we develop an
optimization-based framework for solving FBSDEs by minimizing
the BML value. This framework unifies existing methods in the
literature and extends them to general coupled FBSDEs. Third, we
establish convergence results and error bounds in terms of the
objective function to be optimized.
For general coupled FBSDEs we prove a subsequential convergence result and a one-sided error bound under standard contraction assumptions,
while for the decoupled case we derive two-sided error bounds.
We further validate the proposed objective design and its theoretical interpretation on carefully designed analytical and numerical examples.

\emph{Organizations.} The rest of this paper is organized as
follows.
Section~\ref{section:preliminaries} introduces the notation, the solution space, and the Picard operator in the general coupled FBSDE setting, with the decoupled/BSDE case treated as a special case.
In Section~\ref{section:main-results}, we present the main results of this
paper, including the definition of BML value and its theoretical
properties. In Section~\ref{section:optimization-based-solver}, we describe the
proposed optimization-based framework for solving FBSDEs and show how
it relates to existing methods. Two examples are analytically examined
for demonstration purposes. In Section~\ref{section:experiments}, we
numerically revisit these examples to validate our main
results. We also test the framework on a high-dimensional FBSDE
derived from a nonlinear Hamilton-Jacobi-Bellman (HJB) equation in up to 1000 dimensions. In
Section~\ref{section:conclusion}, we conclude this paper and highlight
future directions. The appendix contains supplementary materials and
technical details omitted in the main text.

\section{Preliminaries}
\label{section:preliminaries}
In this section,
we introduce the notation used throughout the paper,
and formulate the Picard operator $\Phi$ for a coupled FBSDE.
Standard existence and uniqueness results of the considered FBSDE are recalled in the appendix for completeness;
see also \cite{yongStochasticControlsHamiltonian1999,maForwardbackwardStochasticDifferential2007,phamContinuoustimeStochasticControl2009}.

\subsection{Notation}
We adopt notation in the monograph
\cite{maForwardbackwardStochasticDifferential2007}.

\begin{enumerate}
\item Let \(\mu\) denote the Lebesgue measure on the real line.

\item For \(x,y \in \mathbb{R}^{m \times d}\), let \(\langle x, y\rangle:= \operatorname{tr}(x^\intercal y)\) and
\(|x|:=\sqrt{\langle x,x\rangle}\).

\item Let \((\Omega, \mathcal{F}, \mathbb{P})\) be a probability space supporting a standard
\(d\)-dimensional Brownian motion \(W\).

\item Let \(\mathbb{F}=\{\mathcal{F}_t\}_{0\leq t\leq T}\) be the natural filtration generated by
\(W\).

\item Let \(L^2_{\mathcal{G}}(\Omega;N)\) be the set of random variables \(f:\Omega\to N\)
satisfying the following conditions: (i) \(f\) is measurable with
respect to the \(\sigma\)-algebra \(\mathcal{G}\), and (ii) \(f\) is square
integrable, i.e., \(\operatorname{\mathbb{E}}|f|^2 < \infty\).

\item Let \(L^2_{\mathcal{F}}(0,T;N)\) be the set of stochastic processes \(X:\Omega\times[0,T]\to
   N\) satisfying the following conditions: (i) \(X\) is
\(\mathbb{F}\)-progressively measurable, and (ii) \(\int_0^T
   \operatorname{\mathbb{E}}|X_t|^2\,dt < \infty\).

\item Let \(L^2_{\mathcal{F}}(\Omega; C[0,T];N)\) be the set of \emph{continuous} stochastic
processes \(X:\Omega\times[0,T]\to N\) satisfying the following conditions: (i)
\(X\) is \(\mathbb{F}\)-progressively measurable, and (ii)
\(\operatorname{\mathbb{E}}\sup_{t\in[0,T]}|X_t|^2 < \infty\).

\item Let \(L^2_{\mathcal{F}}(0,T;W^{1,\infty}(M;N))\) be the set of functions \(f:\Omega\times[0,T]\times
   M\to N\) satisfying the following conditions: (i) \(f(t,\theta)\) is
Lipschitz continuous with respect to \(\theta\), i.e., there exists a
constant \(L_f\) such that for any \(\theta_1,\theta_2\in M\), the inequality
\(|f(t,\theta_1) - f(t,\theta_2)| \leq L_f|\theta_1 - \theta_2|\) holds almost everywhere on
\(\Omega\times[0,T]\), (ii) \(f\) is \(\mathbb{F}\)-progressively measurable for any fixed
\(\theta\), and (iii) if \(\theta=0\) is fixed, then \(f\in L^2_{\mathcal{F}}(0,T;N)\).

\item Let \(L^2_{\mathcal{F}_T}(\Omega;W^{1,\infty}(M;N))\) be the set of functions \(g:\Omega\times M\to N\)
satisfying the following conditions: (i) \(g(\theta)\) is Lipschitz
continuous with respect to \(\theta\), i.e., there exists a constant \(L_g\)
such that for any \(\theta_1,\theta_2\in M\), the inequality \(|g(\theta_1) - g(\theta_2)| \leq
    L_g|\theta_1 - \theta_2|\) holds almost surely, (ii) \(g\) is \(\mathcal{F}_T\)-measurable
for any fixed \(\theta\), and (iii) if \(\theta=0\) is fixed, then \(g\in
    L^2_{\mathcal{F}_T}(\Omega;N)\).
\end{enumerate}
In the above notation, \(M\) and \(N\) can be any Euclidean spaces with
suitable dimensions.

\subsection{The Picard Operator for FBSDEs}
Consider the general nonlinear FBSDE for \( t\in[0,T]\):
\begin{equation} \left\{
\begin{aligned}
X_t &= x_0 + \int_0^t b(s,X_s,Y_s,Z_s)\,ds + \int_0^t \sigma(s, X_s,Y_s,Z_s)\,dW_s, \\
Y_t &= g(X_T) + \int_t^T f(s,X_s, Y_s,Z_s)\,ds - \int_t^T Z_s\, dW_s,
\end{aligned} \right.
\label{eq:FBSDE}
\end{equation}
where the \emph{forward process} \(X\) is valued in \(\mathbb{R}^n\), the \emph{backward
process} \(Y\) is valued in \(\mathbb{R}^m\),
the \emph{control process} \(Z\) is valued in \(\mathbb{R}^{m \times d}\),
$x_0$ is a constant,
$g$ is a function defined on $\Omega \times \mathbb{R}^n$,
and $b,\sigma,f$ are functions defined on $\Omega\times[0,T]\times \mathbb{R}^n \times \mathbb{R}^m \times \mathbb{R}^{m \times d}$.

\begin{definition}[Solution of FBSDEs]
A triple of processes \((X, Y, Z) \in
L^2_{\mathcal{F}}(\Omega; C[0,T];\mathbb{R}^n) \times L^2_{\mathcal{F}}(\Omega; C[0,T];\mathbb{R}^m) \times L^2_{\mathcal{F}}(0,T;\mathbb{R}^{m \times
d})\) is called an \emph{adapted} solution of \eqref{eq:FBSDE} if for any \(t\in[0,T]\)
the equality holds almost surely.
\end{definition}

\begin{assumption}[Standing Assumption for FBSDEs]
\label{assumption:standing-assumption-for-FBSDEs}
Assume \(g \in
L^2_{\mathcal{F}}(\Omega;W^{1,\infty}(\mathbb{R}^n;\mathbb{R}^m)\)
and $$ \left\{ \begin{aligned} b &\in
  L^2_{\mathcal{F}}(0,T;W^{1,\infty}(\mathbb{R}^n\times\mathbb{R}^m\times\mathbb{R}^{m\times
    d};\mathbb{R}^n),\\ \sigma &\in
  L^2_{\mathcal{F}}(0,T;W^{1,\infty}(\mathbb{R}^n\times\mathbb{R}^m\times\mathbb{R}^{m\times
    d};\mathbb{R}^{m \times d}),\\ f &\in
  L^2_{\mathcal{F}}(0,T;W^{1,\infty}(\mathbb{R}^n\times\mathbb{R}^m\times\mathbb{R}^{m\times
    d};\mathbb{R}^{m}).  \end{aligned} \right.
$$
\end{assumption}

Under Assumption~\ref{assumption:standing-assumption-for-FBSDEs}, the \textbf{Picard operator} \(\Phi\) for
FBSDE~\eqref{eq:FBSDE} can be defined as follows.
For any
\((\tilde{y},\tilde{z})\in L^2_{\mathcal{F}}(\Omega; C[0,T];\mathbb{R}^m) \times
L^2_{\mathcal{F}}(0,T;\mathbb{R}^{m \times d})\), let \(\widetilde{X}\)
satisfy
\begin{equation}
\label{eq:def-widetilde-X}
\widetilde{X}_t = x_0 +
\int_0^t b(s,\widetilde{X}_s,\tilde{y}_s,\tilde{z}_s)\,ds + \int_0^t \sigma(s,
\widetilde{X}_s,\tilde{y}_s,\tilde{z}_s)\,dW_s
\end{equation}
for all $t\in[0,T]$.
Under the current assumption, this SDE admits a unique strong
solution.
Fixing the forward process to $\widetilde{X}$ yields the following linear BSDE,
\begin{equation}
\label{eq:def-widetilde-Y-Z}
\widetilde{Y}_t = g(\widetilde{X}_T) + \int_t^T f(s,\widetilde{X}_s,
\tilde{y}_s,\tilde{z}_s)\,ds - \int_t^T \widetilde{Z}_s\, dW_s.
\end{equation}
By the standard existence and uniqueness result for BSDEs,
the solution of \eqref{eq:def-widetilde-Y-Z} uniquely exists.
Define the Picard operator $\Phi$ by
\begin{equation}
\label{eq:Picard-definition-for-FBSDE}
\Phi(\tilde{y},\tilde{z}):= (\widetilde{Y},\widetilde{Z}),
\end{equation}
where \(\widetilde{Y}\) and \(\widetilde{Z}\) are determined by
\eqref{eq:def-widetilde-X}--\eqref{eq:def-widetilde-Y-Z}.

\textbf{The decoupled case.}
A particularly important special case is the decoupled setting,
in which the forward coefficients $b$ and $\sigma$ in \eqref{eq:FBSDE} do not depend on $(Y,Z)$.
Then the forward SDE can be solved independently,
and substituting the resulting forward process into the backward equation reduces the problem to a BSDE.

\subsection{Norms on the Solution Space}
The optimization variable in our framework is the pair \((Y,Z)\),
while the corresponding forward process is induced by the forward SDE.
We work on the solution space
\[
\mathcal{M}[0,T]
:=
L^2_{\mathcal{F}}(\Omega; C[0,T];\mathbb{R}^m)
\times
L^2_{\mathcal{F}}(0,T;\mathbb{R}^{m \times d}).
\]
This is a Banach space when equipped with the standard norm
\cite[p.~355]{yongStochasticControlsHamiltonian1999}
\begin{equation}
\label{eq:standard-norm}
\|(Y, Z)\| :=
\left\{ \mathbb{E}\sup_{t\in[0,T]}|Y_t|^2 + \mathbb{E}\int_0^T|Z_t|^2\,dt \right\}^{1/2}.
\end{equation}

In addition to the standard norm, we shall also use the following two norms.

\begin{itemize}
\item \emph{The \(\beta\)-norm:} for \(\beta \in \mathbb{R}\),
\begin{equation}
\label{eq:beta-norm}
\|(Y,Z)\|_\beta := \left\{ \mathbb{E} \int_0^T e^{2\beta t}\bigl(|Y_t|^2+|Z_t|^2\bigr)\, dt\right\}^{1/2}.
\end{equation}

\item \emph{The sup-norm:}
\begin{equation}
\label{eq:sup-norm}
\|(Y,Z)\|_\text{sup} :=
\sup_{t\in[0,T]} \left\{ \mathbb{E} |Y_t|^2 + \mathbb{E}\int_t^T |Z_s|^2\,ds
\right\}^{1/2}.
\end{equation}
\end{itemize}

These norms arise naturally in the well-posedness theory of BSDEs and
FBSDEs. In particular, all \(\beta\)-norms are equivalent for different
choices of \(\beta\in\mathbb{R}\), and each of them is weaker than the
standard norm \(\|\cdot\|\). The sup-norm is especially useful in the
analysis of coupled FBSDEs via contraction arguments on small time
intervals. Standard existence and uniqueness results under these norms
are recalled in Appendix~\ref{appendix:existence-results}.

In the next section,
we shall also consider another norm to study the residual interpretation of our objective function.

\section{Main Results}
\label{section:main-results}
In this section, we present our main theoretical results, establishing
a rigorous foundation for an optimization-based framework to solve the
coupled FBSDE~\eqref{eq:FBSDE}.
Specifically, we quantify \emph{how well} a point in the solution space
(called a trial solution) fits a FBSDE by an integral-form value
(called the BML value), and justify it by the fixed point equation of
the Picard operator. Then, we show that any trial solution solves the
considered FBSDE if and only if its BML value equals zero. Moreover,
we prove that a convergent sequence of trial solutions with vanishing
BML values must converge to the true solution. Furthermore, we provide
error bounds to quantify \emph{how close} a trial solution is to the true
solution by its BML value.

Throughout this section, we let
Assumption~\ref{assumption:standing-assumption-for-FBSDEs} hold,
guaranteeing that the Picard operator is well-defined.

\subsection{The BML Value and Its Picard Interpretation}
Let \(\mu\) denote the Lebesgue measure on the real line.

\begin{definition}[BML Value]
For a given trial solution \((\tilde{y},
\tilde{z})\in\mathcal{M}[0,T]\), its BML value for
FBSDE~\eqref{eq:FBSDE} is defined as
\begin{equation}
\label{eq:def-BML}
\operatorname{BML}(\tilde{y},\tilde{z}) := \mathbb{E} \int_0^T
|R_t|^2\,\mu(dt),
\end{equation}
where \(R\) is the backward residual error process (not
necessarily adapted) defined for all \(t\in[0,T]\) by
\begin{equation}
\label{eq:def-residual-error-process}
\left\{ \begin{aligned}
\widetilde{X}_t &= x_0 + \int_0^t
b(s,\widetilde{X}_s,\tilde{y}_s,\tilde{z}_s)\,ds + \int_0^t \sigma(s,
\widetilde{X}_s,\tilde{y}_s,\tilde{z}_s)\,dW_s, \\
R_t &= \tilde{y}_t - \biggl(g(\widetilde{X}_T) + \int_t^T f(s,
\widetilde{X}_s, \tilde{y}_s, \tilde{z}_s)\,ds - \int_t^T \tilde{z}_s
\,dW_s \biggr).
\end{aligned} \right.
\end{equation}
\end{definition}
This definition evaluates the coupled FBSDE through a reduced parametrization:
for each trial pair $(\tilde y,\tilde z)$,
the associated forward process $\widetilde X$ is first generated by the forward SDE,
and the backward residual $R$ is then evaluated using $(\widetilde X, \tilde y, \tilde z)$.
The process \(\widetilde{X}\) is same as the process defined by \eqref{eq:def-widetilde-X},
which captures the forward equation of \eqref{eq:FBSDE}.
The process \(R\) captures the backward equation of \eqref{eq:FBSDE},
quantifying the residual error after replacing $(Y,Z)$ with $(\tilde y, \tilde z)$ in \eqref{eq:def-widetilde-Y-Z}.
Given $(\tilde{y},\tilde{z})$,
the BML value could be estimated using Monte Carlo methods after time discretization;
see Section~\ref{section:optimization-based-solver} for more details on calculating BML values.

Intuitively, the BML value quantifies how well a trial solution fits
FBSDE~\eqref{eq:FBSDE}. Indeed, if the residual error process \(R\) is
almost everywhere zero, then the trial solution is expected to solve
the FBSDE. The following theorem, however, provides another
insightful way to interpret it through the Picard operator \(\Phi\)
defined for this FBSDE.

\begin{theorem}[The Picard Interpretation of BML Values]
\label{theorem:the-Picard-interpretations-of-BML-values}
For any trial solution \((\tilde{y}, \tilde{z})\in\mathcal{M}[0,T]\),
its BML value for FBSDE~\eqref{eq:FBSDE} equals the residual loss of
the fixed point equation of \(\Phi\), i.e.,
\begin{equation}
\label{eq:the-Picard-interpretations-of-BML-values}
\operatorname{BML}(\tilde{y},\tilde{z}) = \|(\tilde{y},\tilde{z}) -
\Phi(\tilde{y},\tilde{z})\|^2_\mu.
\end{equation}
Here, \(\|\cdot\|_\mu\) is defined by
\begin{equation}
\label{eq:mu-norm}
\|(Y, Z)\|_\mu := \left\{ \mathbb{E} \int_0^T \biggl( |Y_t|^2 + \int_t^T |Z_s|^2 \,ds
\biggr)\, \mu(dt)\right\}^{1/2}.
\end{equation}
\end{theorem}

\begin{proof}
Let \(\widetilde{X}\) be the process in
\eqref{eq:def-residual-error-process}. Let \((\widetilde{Y},
\widetilde{Z}):=\Phi(\tilde{y},\tilde{z})\). Then, the triple
\((\widetilde{X},\widetilde{Y},\widetilde{Z})\) solves the following
linear decoupled FBSDE $$ \left\{ \begin{aligned} \widetilde{X}_t &=
  x_0 + \int_0^t b(s,\widetilde{X}_s,\tilde{y}_s,\tilde{z}_s)\,ds +
  \int_0^t \sigma(s, \widetilde{X}_s,\tilde{y}_s,\tilde{z}_s)\,dW_s,
  \\ \widetilde{Y}_t &= g(\widetilde{X}_T) + \int_t^T
  f(s,\widetilde{X}_s,\tilde{y}_s,\tilde{z}_s)\,ds - \int_t^T
  \widetilde{Z}_s dW_s.  \end{aligned}\right. $$ Therefore, the
residual error process becomes $$ R(t,\omega;\tilde{y},\tilde{z}) =
\tilde{y}_t - \biggl(\widetilde{Y}_t + \int_t^T\widetilde{Z}_s\,dW_s -
\int_t^T \tilde{z}_s \,dW_s \biggr). $$ Taking expectation yields
$$ \begin{aligned}
&\mathbb{E}\left[|R(t,\omega;\tilde{y},\tilde{z})|^2\right] \\
&=
\mathbb{E}\left[|\tilde{y}_t - \widetilde{Y}_t|^2\right] + \mathbb{E}\left[\left(\int_t^T
(\tilde{z}_s - \widetilde{Z}_s)\,dW_s \right)^2\right] \\
&\quad+ \mathbb{E}\left[
(\tilde{y}_t - \widetilde{Y}_t) \left(\int_t^T (\tilde{z}_s -
\widetilde{Z}_s)\,dW_s \right) \right] \\
&= \mathbb{E}\left[|\tilde{y}_t -
\widetilde{Y}_t|^2\right] + \mathbb{E}\int_t^T |\tilde{z}_s -
\widetilde{Z}_s|^2\,ds.  \end{aligned} $$ Taking integrals on both
sides over \([0,T]\) completes the proof.
\end{proof}

\begin{remark}
\label{remark:choices-of-mu-measure}
This proof remains valid for any finite measure \(\mu\). Indeed, it is
possible to extend our results to other measures that are equivalent
to the Lebesgue measure, providing flexibility in practical
implementations. Nevertheless, we fix \(\mu\) to the standard Lebesgue
measure to avoid those technical details in this work.
\end{remark}

This theorem underpins the main idea of the paper:
the fixed-point equation can be addressed by directly minimizing the BML value.
Instead of explicitly evaluating the Picard operator,
which is computationally expensive because it involves solving a linear FBSDE,
one may directly minimize the BML value with respect to the decision variables \((\tilde{y},\tilde{z})\).

The \(\|\cdot\|_\mu\), referred to as the \(\mu\)-norm, is indeed a norm on the
solution space, and is weaker than all the three norms introduced
before.

\begin{lemma}[\(\mu\)-norm]
\label{lemma:mu-norm}
The \(\|\cdot\|_\mu\) defined in \eqref{eq:mu-norm} can be
equivalently written as
\begin{equation}
\label{eq:mu-norm-alternative-def}
\|(Y, Z)\|_\mu := \left\{ \mathbb{E} \int_0^T
\biggl( |Y_t|^2 + t|Z_t|^2 \biggr)\, dt\right\}^{1/2}.
\end{equation}
It is a norm on \(\mathcal{M}[0,T]\).
Moreover,
it is weaker than the standard
norm~\eqref{eq:standard-norm},
the \(\beta\)-norm \eqref{eq:beta-norm},
and the sup-norm~\eqref{eq:sup-norm}.
\end{lemma}

\begin{proof}
See Appendix~\ref{appendix:norms-on-the-solution-space}.
\end{proof}

\begin{remark}
The equivalence between \eqref{eq:mu-norm} and \eqref{eq:mu-norm-alternative-def} relies on the
fact that \(\mu([0,t])=t\), as we fix \(\mu\) to the Lebesgue measure. In
general, for any finite measure \(\mu\) on \([0,T]\), the definition \eqref{eq:mu-norm} is
equivalent to $$ \|(Y, Z)\|_\mu := \left\{ \mathbb{E} \int_0^T |Y_t|^2 \,\mu(dt) + \mathbb{E}
\int_0^T \mu([0,s])|Z_s|^2 \, ds\right\}^{1/2}.  $$
\end{remark}

Combining Lemma~\ref{lemma:mu-norm} and
Theorem~\ref{theorem:the-Picard-interpretations-of-BML-values}
concludes that a trial solution has zero BML value if and only if it
is a fixed point of \(\Phi\).

\begin{corollary}[Zero BML Value Solution]
\label{corollary:zero-BML-value-solution}
Let \((\tilde{y},\tilde{z})\) be a point in the solution space
\(\mathcal{M}[0,T]\). Then, its BML value for FBSDE~\eqref{eq:FBSDE}
equals zero if and only if it is a part of an adapted solution of that
FBSDE.
\end{corollary}


Corollary~\ref{corollary:zero-BML-value-solution} shows that solving
FBSDE~\eqref{eq:FBSDE} is equivalent to driving the BML value to zero.
For numerical optimization, however, one typically obtains trial
solutions with small but nonzero BML values. This raises two natural
questions: whether vanishing BML values force convergence to the true
solution, and how the BML value quantitatively controls the solution
error. We address these questions next, first in the general coupled
setting and then, in the subsequent subsection, in the decoupled case
where sharper estimates are available.

\subsection{Convergence Analysis and Error Bounds with BML Values}
\label{sec:error-bounds}
In this subsection,
we study to what extent the BML value controls the
distance between a trial solution and the true solution.
We begin with the general coupled FBSDE case, which is the primary
setting of this paper. When the forward and backward equations are
decoupled, the associated Picard operator enjoys better contraction properties,
and strong conclusions are available under the standard contraction assumptions.
Sharper estimates in the special case will be derived later.

First,
we establish a useful lemma to reveal relationship between the contraction property of $\Phi$ and the
\(\mu\)-norm of BML.

\begin{lemma}[Contraction Under a Norm Stronger Than \(\mu\)-norm]
\label{lemma:contraction-under-a-norm-stronger-than-mu-norm}
Let Assumption~\ref{assumption:standing-assumption-for-FBSDEs} hold and additional
assumptions in Lemma~\ref{lemma:FBSDE-Existence} hold. Then, there exists a
norm on \(\mathcal{M}[0,T]\), denoted by \(\|\cdot\|_{\bar{\mu}}\),
which satisfies: 1) it is stronger than the \(\mu\)-norm but weaker
than the standard norm; 2) the Picard operator \(\Phi\) is a strict
contraction under \(\|\cdot\|_{\bar{\mu}}\).
\end{lemma}

\begin{proof}
By the standard existence result of FBSDEs (Lemma~\ref{lemma:FBSDE-Existence} in Appendix~\ref{appendix:existence-results}),
the Picard operator \(\Phi\) is a strict contraction under the sup-norm~\eqref{eq:sup-norm}.
In view of Lemma~\ref{lemma:mu-norm},
the desired norm \(\|\cdot\|_{\bar{\mu}}\) can be chosen as the sup-norm.
\end{proof}

\begin{theorem}[Convergence of BML Values Implies Convergence to the True Solution]
\label{theorem:convergence-of-BML-values-implies-convergence-to-the-true-solution-coupled-case}
Let Assumption~\ref{assumption:standing-assumption-for-FBSDEs} hold
and additional assumptions in Lemma~\ref{lemma:FBSDE-Existence}
hold. Consider a sequence \(\{(\tilde{y}^{(k)},
\tilde{z}^{(k)})\}_{k=1}^\infty\) in the solution space
\(\mathcal{M}[0,T]\)
whose associated sequence of BML values for
FBSDE~\eqref{eq:FBSDE} converges to zero,
i.e., $$ \lim_{k\to\infty} \operatorname{BML}(\tilde{y}^{(k)},
\tilde{z}^{(k)}) = 0. $$
Then, any convergent
subsequence of \(\{(\tilde{y}^{(k)}, \tilde{z}^{(k)})\}_{k=1}^\infty\)
converges to the true solution \((Y^*,Z^*)\), where the convergence of
trial solutions is understood in the sense of the standard norm on
\(\mathcal{M}[0,T]\).
\end{theorem}

\begin{proof}
Let \((\tilde{y}^*,
\tilde{z}^*)\in\mathcal{M}[0,T]\) be an accumulation point of \(\{(\tilde{y}^{(k)},
\tilde{z}^{(k)})\}_{k=1}^\infty\) under the standard norm. Without
loss of generality, assume that the whole sequence converges to
\((\tilde{y}^*, \tilde{z}^*)\). Otherwise, we replace the original
sequence with its convergent subsequence and apply the same proof.

Let \(\|\cdot\|_{\bar{\mu}}\) be the norm described in
Lemma~\ref{lemma:contraction-under-a-norm-stronger-than-mu-norm}.
Then,
the trial solution sequence \(\{(\tilde{y}^{(k)},
\tilde{z}^{(k)})\}_{k=1}^\infty\) is also convergent under \(\|\cdot\|_{\bar{\mu}}\)
and \(\mu\)-norm.

It is sufficient to show that the limit \((\tilde{y}^*, \tilde{z}^*)\)
achieves zero BML value. For any \(k \geq 1\),
$$ \begin{aligned}
&\|(\tilde{y}^*, \tilde{z}^*) - \Phi(\tilde{y}^*, \tilde{z}^*)\|_\mu\\
&\leq\|(\tilde{y}^*, \tilde{z}^*)-(\tilde{y}^{(k)},
\tilde{z}^{(k)})\|_\mu+\|(\tilde{y}^{(k)},
\tilde{z}^{(k)})-\Phi(\tilde{y}^{(k)},
\tilde{z}^{(k)})\|_\mu\\
&\quad+\|\Phi(\tilde{y}^{(k)},
\tilde{z}^{(k)})-\Phi(\tilde{y}^*, \tilde{z}^*)\|_\mu \\
&\leq  C\|(\tilde{y}^*, \tilde{z}^*)-(\tilde{y}^{(k)},
\tilde{z}^{(k)})\|_{\bar{\mu}}+\|(\tilde{y}^{(k)},
\tilde{z}^{(k)})-\Phi(\tilde{y}^{(k)},
\tilde{z}^{(k)})\|_\mu\\
&\quad+C\|\Phi(\tilde{y}^{(k)},
\tilde{z}^{(k)})-\Phi(\tilde{y}^*, \tilde{z}^*)\|_{\bar{\mu}},
\end{aligned} $$ where the second inequality comes from the fact that
\(\|\cdot\|_\mu \leq C \|\cdot\|_{\bar{\mu}}\).
On the right hand side, the
first term vanishes for large enough \(k\) as \(\{(\tilde{y}^{(k)},
\tilde{z}^{(k)})\}_{k=1}^\infty\) converges to \((\tilde{y}^*,
\tilde{z}^*)\) under the \(\mu\)-norm. The second term vanishes for
large enough \(k\) as the associated BML values converge to zero (applying
Theorem~\ref{theorem:the-Picard-interpretations-of-BML-values}). The
third term also vanishes for large enough \(k\) as \(\Phi\) is
continuous under the \(\mu\)-norm. Therefore,
$$ \|(\tilde{y}^*, \tilde{z}^*) - \Phi(\tilde{y}^*, \tilde{z}^*)\|_\mu =
0.$$
Applying Corollary~\ref{corollary:zero-BML-value-solution} concludes the proof.
\end{proof}


Note that Lemma~\ref{lemma:contraction-under-a-norm-stronger-than-mu-norm}
only asserts that \(\Phi\) is a contraction under a norm stronger than the
\(\mu\)-norm, not a norm equivalent to the \(\mu\)-norm. For this reason, we currently
can only obtain one direction of the error bound for coupled FBSDEs.

\begin{proposition}[A Lower Error Bound for Coupled FBSDEs]
\label{proposition:a-lower-error-bound-for-coupled-FBSDEs}
Let Assumption~\ref{assumption:standing-assumption-for-FBSDEs} hold and additional
assumptions in Lemma~\ref{lemma:FBSDE-Existence} hold. Let \((\tilde{y},\tilde{z}) \in
\mathcal{M}[0,T]\).  Then, the error between it and the
true solution \((Y^*,Z^*)\) under the standard norm can be bounded below
by its BML values: there exists a positive constant \(C\), independent
of \((\tilde{y},\tilde{z})\), such that $$ C
\operatorname{BML}(\tilde{y},\tilde{z}) \leq
\|(\tilde{y},\tilde{z})-(Y^*,Z^*)\|^2. $$
\end{proposition}

\begin{proof}
Let \(\|\cdot\|_{\bar{\mu}}\) be the norm described in Lemma~\ref{lemma:contraction-under-a-norm-stronger-than-mu-norm}. Then, for all \((Y, Z) \in \mathcal{M}[0,T]\), there exist
two positive constant \(c_1\) and \(c_2\) such that
\begin{equation}
\label{eq:tmp-466}
c_1 \|(Y, Z)\|_\mu \leq \|(Y, Z)\|_{\bar{\mu}} \leq c_2 \|(Y, Z)\|.
\end{equation}
Moreover, for all \((Y, Z),(\hat{Y},\hat{Z}) \in \mathcal{M}[0,T]\), there exists a constant \(L > 0\) such that
$$ \|\Phi(Y,Z) -
\Phi(\hat{Y},\hat{Z})\|_{\bar{\mu}} \leq L\|(Y,Z) - (\hat{Y},\hat{Z})\|_{\bar{\mu}}.$$
For
\((\tilde{y},\tilde{z}) \in \mathcal{M}[0,T]\), there is
\begin{equation}
\label{eq:tmp-474}
\begin{aligned}
&\|(\tilde{y},\tilde{z})-\Phi(\tilde{y},\tilde{z})\|_{\bar{\mu}}\\ &\leq\|(\tilde{y},\tilde{z})-(Y^*,Z^*)\|_{\bar{\mu}}+\|\Phi(Y^*,Z^*)-\Phi(\tilde{y},\tilde{z})\|_{\bar{\mu}}\\
&\quad+\|(Y^*,Z^*)-\Phi(Y^*,Z^*)\|_{\bar{\mu}}\\
&\leq(1+L)\|(\tilde{y},\tilde{z})-(Y^*,Z^*)\|_{\bar{\mu}}.
\end{aligned}
\end{equation}
The proof is concluded by combining \eqref{eq:tmp-466} and
\eqref{eq:tmp-474} and applying
Theorem~\ref{theorem:the-Picard-interpretations-of-BML-values}.
\end{proof}

Note that this proof only uses the Lipschitz continuity of
\(\Phi\). Therefore, the smallness of \(T\) might be dropped at the cost of
\(\Phi\) no longer being a strict contraction. However, that would require
other assumptions to ensure the existence and uniqueness of the
solution \((Y^*,Z^*)\).

The above results provide convergence and a one-sided error estimate in
the general coupled setting. We next turn to the decoupled case,
equivalently the BSDE setting after solving the forward equation
independently, where the Picard operator admits stronger norm
properties and hence sharper conclusions, including two-sided error
bounds.

\subsection{Two-Sided Error Bounds in the Decoupled Case}
\label{sec:two-sided-error-bounds-in-the-decoupled-case}
We now consider the decoupled case, where the forward equation can
be solved independently of the backward equation. After substituting
the resulting forward process into the backward equation, the problem
reduces to a BSDE. 
For ease of notation,
we therefore drop the forward equation of \eqref{eq:FBSDE} and focus here on the BSDE
\begin{equation}
\label{eq:BSDE}
Y_t = \xi + \int_t^T f(s,Y_s,Z_s)\,ds - \int_t^T Z_s\, dW_s, \quad t\in[0,T],
\end{equation}
where the \emph{backward process} \(Y\) is valued in \(\mathbb{R}^m\),
the \emph{control process} \(Z\) is valued in \(\mathbb{R}^{m \times d}\),
$\xi$ is a random variable,
and $f$ is a function defined on $\Omega\times[0,T]\times \mathbb{R}^m \times \mathbb{R}^{m \times d}$.

The standing assumption of \eqref{eq:BSDE} is given below.
\begin{assumption}[Standing Assumption for BSDEs]
\label{assumption:standing-assumption-for-BSDEs}
Assume \(\xi \in L^2_{\mathcal{F}_T}(\Omega;\mathbb{R}^m)\) and \(f
\in L^2_{\mathcal{F}}(0,T;W^{1,\infty}(\mathbb{R}^m \times
\mathbb{R}^{m \times d};\mathbb{R}^m))\).
\end{assumption}
The Picard operator $\Phi$ can be similarly defined,
and its well-posedness is guaranteed under Assumption~\ref{assumption:standing-assumption-for-BSDEs}.
Lemma~\ref{lemma:BSDE-Existence} in Appendix~\ref{appendix:existence-results}
guarantees the existence and uniqueness of the true solution
\((Y^*,Z^*)\), and shows that the Picard operator \(\Phi\) is a strict
contraction under a suitable \(\beta\)-norm~\eqref{eq:beta-norm}.
Although the \(\mu\)-norm~\eqref{eq:mu-norm} is not equivalent to the
\(\beta\)-norm, one can construct a norm equivalent to the
\(\mu\)-norm under which \(\Phi\) remains a strict contraction.

\begin{lemma}[Contraction Under a Norm Equivalent to \(\mu\)-norm]
\label{lemma:contraction-under-a-norm-equivalent-to-mu-norm}
Let Assumption~\ref{assumption:standing-assumption-for-BSDEs}
hold. Then, there exists a norm on \(\mathcal{M}[0,T]\), denoted by
\(\|\cdot\|_{\mu(\beta)}\), which satisfies: 1) it is equivalent to
the \(\mu\)-norm; 2) the Picard operator \(\Phi\) is a strict
contraction under \(\|\cdot\|_{\mu(\beta)}\).
\end{lemma}

\begin{proof}
The desired norm can be chosen as
\[
\|(Y,Z)\|_{\mu(\beta)}^2
:=
\mathbb E\int_0^T e^{2\beta t}\bigl(|Y_t|^2 + t|Z_t|^2\bigr)\,dt,
\]
for some constant $\beta \in \mathbb{R}$;
see Appendix~\ref{appendix:contraction-property-of-Picard-operator-for-BSDEs}.
\end{proof}

In particular, this lemma implies that \(\Phi\) is continuous
(in fact, Lipschitz continuous) under the \(\mu\)-norm.
Consequently, one can relate the error of a trial solution directly to its
BML value by means of two-sided bounds.

\begin{theorem}[Two-Sided Error Bounds for BSDEs]
\label{theorem:error-bounds-with-BML-values-for-BSDEs}
Let Assumption~\ref{assumption:standing-assumption-for-BSDEs}
hold. Consider a trial solution \((\tilde{y},\tilde{z}) \in
\mathcal{M}[0,T]\). Then, the error between it and the true solution
\((Y^*,Z^*)\) under \(\mu\)-norm can be bounded by its BML value:
there exist positive constants \(C_1\) and \(C_2\), independent of
\((\tilde{y},\tilde{z})\), such that $$ C_1
\operatorname{BML}(\tilde{y},\tilde{z}) \leq
\|(\tilde{y},\tilde{z})-(Y^*,Z^*)\|_\mu^2 \leq C_2
\operatorname{BML}(\tilde{y},\tilde{z}).$$
\end{theorem}

\begin{proof}
Let \(\|\cdot\|_{\mu(\beta)}\) be the norm described in
Lemma~\ref{lemma:contraction-under-a-norm-equivalent-to-mu-norm}. Then, for all \((Y, Z) \in \mathcal{M}[0,T]\),
there exist positive constants \(c_1\) and \(c_2\) such that
\begin{equation}
\label{eq:tmp-323}
c_1 \|(Y, Z)\|_\mu \leq \|(Y, Z)\|_{\mu(\beta)} \leq c_2\|(Y, Z)\|_\mu.
\end{equation}
Moreover,
there exists a constant \(L \in (0, 1)\) such that for all \((Y, Z),(\hat{Y},\hat{Z}) \in \mathcal{M}[0,T]\),
$$ \|\Phi(Y,Z) -
\Phi(\hat{Y},\hat{Z})\|_{\mu(\beta)} \leq L\|(Y,Z) - (\hat{Y},\hat{Z})\|_{\mu(\beta)}.$$
For
\((\tilde{y},\tilde{z}) \in \mathcal{M}[0,T]\), there is
\begin{equation}
\label{eq:tmp-331}
\begin{aligned}
&\|(\tilde{y},\tilde{z})-\Phi(\tilde{y},\tilde{z})\|_{\mu(\beta)}\\ &\leq
\|(\tilde{y},\tilde{z})-(Y^*,Z^*)\|_{\mu(\beta)}+\|\Phi(Y^*,Z^*)-\Phi(\tilde{y},\tilde{z})\|_{\mu(\beta)}\\
&\quad+\|(Y^*,Z^*)-\Phi(Y^*,Z^*)\|_{\mu(\beta)}\\
&\leq(1+L)\|(\tilde{y},\tilde{z})-(Y^*,Z^*)\|_{\mu(\beta)}.
\end{aligned}
\end{equation}
In the other direction, there is
\begin{equation}
\label{eq:tmp-339}
\begin{aligned}
&\|(\tilde{y},\tilde{z})-(Y^*,Z^*)\|_{\mu(\beta)} \\
&\leq\|(\tilde{y},\tilde{z})-\Phi(\tilde{y},\tilde{z})\|_{\mu(\beta)}+\|\Phi(\tilde{y},\tilde{z})-\Phi(Y^*,Z^*)\|_{\mu(\beta)}\\
&\quad+\|(Y^*,Z^*)-\Phi(Y^*,Z^*)\|_{\mu(\beta)}\\
&\leq\|(\tilde{y},\tilde{z})-\Phi(\tilde{y},\tilde{z})\|_{\mu(\beta)}+L\|(\tilde{y},\tilde{z})-(Y^*,Z^*)\|_{\mu(\beta)}.
\end{aligned}
\end{equation}
Combining \eqref{eq:tmp-331} and \eqref{eq:tmp-339} yields
\begin{equation}
\label{eq:tmp-348}
\begin{split}
    \frac{1}{1+L}\|(\tilde{y},\tilde{z})-\Phi(\tilde{y},\tilde{z})\|_{\mu(\beta)}&\leq\|(\tilde{y},\tilde{z})-(Y^*,Z^*)\|_{\mu(\beta)}\\
    &\leq\frac{1}{1-L}\|(\tilde{y},\tilde{z})-\Phi(\tilde{y},\tilde{z})\|_{\mu(\beta)}.
\end{split}
\end{equation}
The proof is concluded by combining \eqref{eq:tmp-323} and
\eqref{eq:tmp-348} and applying Theorem~\ref{theorem:the-Picard-interpretations-of-BML-values}.
\end{proof}

Theorem~\ref{theorem:error-bounds-with-BML-values-for-BSDEs} has an important consequence.
Unlike the earlier
convergence statement for FBSDEs in the coupled case, the two-sided
error bounds here show that, in the BSDE setting, one does \emph{not}
need to assume the existence of a convergent subsequence. Indeed, if a
sequence of trial solutions has vanishing BML values, then the sequence
itself converges automatically;
moreover, its limit must be the true solution.

\begin{corollary}
\label{corollary:convergence-of-BML-values-implies-convergence-to-the-true-solution}
Let Assumption~\ref{assumption:standing-assumption-for-BSDEs}
hold. Consider a sequence \(\{(\tilde{y}^{(k)},
\tilde{z}^{(k)})\}_{k=1}^\infty\) in the solution space
\(\mathcal{M}[0,T]\) whose associated sequence of BML values for
BSDE~\eqref{eq:BSDE} converges to zero, i.e.,
$$
\lim_{k\to\infty} \operatorname{BML}(\tilde{y}^{(k)}, \tilde{z}^{(k)}) = 0.
$$
Then the sequence \(\{(\tilde{y}^{(k)}, \tilde{z}^{(k)})\}_{k=1}^\infty\)
itself converges to the true solution \((Y^*,Z^*)\) of that BSDE in the
\(\mu\)-norm.
\end{corollary}


\begin{remark}
This corollary strengthens the corresponding convergence statement proved earlier
for the coupled FBSDE case. There, vanishing BML values only imply that
\emph{any convergent subsequence} must converge to the true solution.
Here, by contrast, the two-sided error bounds imply convergence of the
entire sequence directly, with no a priori compactness or subsequence
assumption.

\end{remark}


Taken together, Section~\ref{sec:error-bounds} and
Section~\ref{sec:two-sided-error-bounds-in-the-decoupled-case} show that the BML framework applies directly to
general coupled FBSDEs, while the decoupled setting admits
sharper quantitative estimates.
\section{Optimization-Based Solver for FBSDEs}
\label{section:optimization-based-solver}
This section presents the practical optimization-based framework
corresponding to the BML theory developed in the previous section.
We first describe how the continuous-time objective is approximated by
time discretization and Monte Carlo sampling. Then,
we introduce the
parameterization scheme adopted in this paper, in which the backward
process \(Y\) and the control process \(Z\) are modeled
simultaneously. The relation to existing methods and analytical
examples will be discussed in the subsequent subsections.

\subsection{Time Discretization}
The theoretical results in Section~\ref{section:main-results} suggest
the following optimization problem for solving
FBSDE~\eqref{eq:FBSDE}:
\begin{equation}
\label{eq:def-optimization-problem}
\min_{(\tilde{y},\tilde{z})} \frac{1}{T}
\operatorname{BML}(\tilde{y},\tilde{z}),
\end{equation}
where \(\operatorname{BML}(\tilde{y},\tilde{z})\) is defined by
\eqref{eq:def-BML}--\eqref{eq:def-residual-error-process}. By
Corollary~\ref{corollary:zero-BML-value-solution}, a trial
solution solves the FBSDE if and only if its BML value equals zero.
Moreover, Section~\ref{section:main-results} shows that small BML
values are theoretically meaningful, as they imply closeness to the
true solution under the corresponding assumptions.
If solving optimization
problem~\eqref{eq:def-optimization-problem} yields a sequence of
trial solutions with vanishing objective values, then
Theorem~\ref{theorem:convergence-of-BML-values-implies-convergence-to-the-true-solution-coupled-case}
asserts that any convergent subsequence of it must converge to the true
solution. If numerically solving this optimization problem yields a
trial solution with a sufficiently small objective value, then
Proposition~\ref{proposition:a-lower-error-bound-for-coupled-FBSDEs}
and Theorem~\ref{theorem:error-bounds-with-BML-values-for-BSDEs}
provide error bounds for estimating the distance between the obtained
solution and the true solution.

We next explain how
the continuous-time objective
\eqref{eq:def-optimization-problem} is approximated numerically.
By definition, the BML value can
be estimated by Monte Carlo simulations for \(R(t,\omega)\), where
\(R\) is the backward residual error process defined in
\eqref{eq:def-residual-error-process}.  Let $M$ be the number of Monte
Carlo samples and $H$ be the number of time intervals for time
discretization. First, generate a collection of Brownian motion paths
on the time grid
$$ \mathbb{T}:=\{t_i; 0 \leq i \leq H\}, \quad \text{ where } t_i:=
i\,\Delta t:= \frac{iT}{H}. $$ Then, simulate the forward
SDE~\eqref{eq:def-widetilde-X} via the Euler-Maruyama method to obtain
collections of paths on the same time grid \(\mathbb{T}\) for the
triple \((\widetilde{X},\tilde{y},\tilde{z})\). Moreover, the
collection of paths for \(R\) on the same time grid \(\mathbb{T}\) can
be obtained by
\begin{equation}
\label{eq:discrete-residual}
\begin{aligned}
    R_{t_i} := \tilde{y}_{t_i} -
\biggl(&g(\widetilde{X}_{T}) + \sum_{k=i}^{H-1} f(t_k,
\widetilde{X}_{t_k}, \tilde{y}_{t_k}, \tilde{z}_{t_k})\,\Delta t\biggr.\\
&\biggl.-\sum_{k=i}^{H-1} \tilde{z}_{t_k}\,(W_{t_{k+1}} - W_{t_k}) \biggr),
\quad 0 \leq i \leq H.
\end{aligned}
\end{equation}
To estimate the integral in the BML value, we adopt a particle-style
Monte Carlo approximation.
For a particular sample path
\(R(\cdot,\omega^{(j)})\), randomly select a time instant
\(t_j\in\mathbb{T}\) and return \(|R(t_j,\omega^{(j)})|^2\) as one sample of the objective.
Finally, taking the empirical
expectation over all samples gives the estimation
\begin{equation}
\label{eq:approximate-optimization-problem}
\frac{1}{T}
\operatorname{BML}(\tilde{y},\tilde{z}) \approx \frac{1}{M}
\sum_{j=1}^M |R(t_j,\omega^{(j)})|^2.
\end{equation}
In principle,
a reasonably good Monte Carlo estimation with small
confidence-intervals may require large enough \(M\), e.g., \(10^6\) or
even \(10^9\). Nevertheless, when numerically minimizing via the
stochastic gradient descent (SGD) method and its variants, we could
use a significantly small \(M\), e.g., \(10^3\).

\subsection{Parameterization}
\begin{algorithm}[t]
\caption{Optimization-Based Solver with BML}
\label{alg:parallel-parameterization}
\begin{algorithmic}[1]
\STATE \textbf{Input:} initial parameter vector \(\theta\), learning rate \(\alpha\), batch size \(M\), number of time intervals \(H\), total number of iterations \(K\)
\STATE Construct the time grid
\(
\mathbb{T}=\{t_i=iT/H;\;0\le i\le H\}
\)
\FOR{\(k=1,2,\ldots,K\)}
    \STATE Sample \(M\) independent Brownian motion paths on \(\mathbb{T}\)
    \FOR{each sampled path \(\omega^{(j)}\), \(j=1,\ldots,M\)}
        \STATE Simulate \(\widetilde{X}\) on \(\mathbb{T}\) by the Euler--Maruyama method using \eqref{eq:parallel-parameterization-scheme}
        \STATE Evaluate
        \(
        \tilde{y}_{t_i}=y^\theta(t_i,\widetilde{X}_{t_i})
        \)
        and
        \(
        \tilde{z}_{t_i}=z^\theta(t_i,\widetilde{X}_{t_i})
        \)
        on the grid
        \STATE Compute the discrete residuals
        \(
        \{R_{t_i}\}_{i=0}^H
        \)
        by \eqref{eq:discrete-residual}
        \STATE Randomly select one time instant \(t_j\in\mathbb{T}\) and record the particle value
        \[
        \ell^{(j)}(\theta):=|R(t_j,\omega^{(j)})|^2
        \]
    \ENDFOR
    \STATE Compute the mini-batch empirical loss
    \[
    \widehat{\mathcal L}(\theta):=
    \frac{1}{M}\sum_{j=1}^M \ell^{(j)}(\theta)
    =
    \frac{1}{M}\sum_{j=1}^M |R(t_j,\omega^{(j)})|^2
    \]
    \STATE Update the parameter by stochastic gradient descent:
    \[
    \theta \leftarrow \theta-\alpha \nabla_\theta \widehat{\mathcal L}(\theta)
    \]
\ENDFOR
\STATE \textbf{Output:} trained parameterized functions \(y^\theta\) and \(z^\theta\)
\end{algorithmic}
\end{algorithm}

To solve optimization problem
\eqref{eq:def-optimization-problem} numerically, the trial pair
\((\tilde{y},\tilde{z})\) must be restricted to a parameterized family.
In this paper, we adopt a parallel parameterization scheme, where
the backward process \(Y\) and the control process \(Z\) are modeled
simultaneously.

Specifically, let \(y^\theta\) and \(z^\theta\) be two parameterized
functions. Given a parameter vector \(\theta\), the trial solution is
constructed by
\begin{equation}
\label{eq:parallel-parameterization-scheme}
\left\{
\begin{aligned}
\widetilde{X}_t
=&\;
x_0
+\int_0^t
b\bigl(s,\widetilde{X}_s,y^\theta(s,\widetilde{X}_s),
z^\theta(s,\widetilde{X}_s)\bigr)\,ds
\\
&\;
+\int_0^t
\sigma\bigl(s,\widetilde{X}_s,y^\theta(s,\widetilde{X}_s),
z^\theta(s,\widetilde{X}_s)\bigr)\,dW_s,
\\
\tilde{y}_t
=&\;
y^\theta(t,\widetilde{X}_t),
\qquad
\tilde{z}_t
=
z^\theta(t,\widetilde{X}_t).
\end{aligned}
\right.
\end{equation}
Thus,
$\widetilde X$ is not parameterized independently;
it is generated from the forward SDE once the trial pair $(\tilde y,\tilde z)$.
The same parameter vector \(\theta\) determines the
entire trial triple \((\widetilde{X},\tilde{y},\tilde{z})\). The BML
objective therefore becomes a finite-dimensional function of
\(\theta\), after time discretization and Monte Carlo approximation.

Instead of prescribing only one component of the backward equation,
this parameterization scheme models \(Y\) and \(Z\) in parallel and lets the forward
process be induced through the forward SDE. Such a design is
particularly natural for coupled FBSDEs, where the forward and
backward equations interact through both processes.
This is a core feature of our optimization-based framework.

In practical implementations, the functions \(y^\theta\) and
\(z^\theta\) can be chosen from general approximation classes, such as
neural networks. The particular architectures used in our numerical
experiments will be specified in
Section~\ref{section:experiments}. At the current level, it is enough
to regard \(y^\theta\) and \(z^\theta\) as generic parameterized
approximators of the two unknown processes.
The resulting computation pipeline is summarized in Algorithm~\ref{alg:parallel-parameterization}.

\subsection{Relation to Existing Methods}
The proposed solver is formulated for general coupled
FBSDEs. Under suitable choices of parameterization and
time discretization, its objective reduces to several representative
optimization-based methods in the literature. In this sense, the BML
framework provides a common objective-based viewpoint for these
methods, while remaining applicable to the more general coupled
setting.

The methods discussed in this subsection are all optimization-based approaches, and are therefore the most closely related to the present work. They share a common feature with our framework in that solving the FBSDE is reformulated as minimizing a loss or objective function over a parameterized class of trial solutions. In this sense, they have a similar probabilistic and computational nature, even though the specific parameterizations, objective designs, and theoretical analyses are different. By contrast, PDE-based approaches, as already mentioned in the introduction, are also valid for solving FBSDEs through the (nonlinear) Feynman--Kac correspondence. However, such methods are conceptually and numerically quite different from the current framework, since they proceed by first reformulating the problem as an associated PDE and then applying PDE solvers, often involving spatial discretization or other PDE-specific numerical techniques. For this reason, the discussion below focuses on optimization-based methods that are more directly comparable to our approach.

\textbf{Relation to the deep BSDE method
\cite{hanSolvingHighdimensionalPartial2018}.} For each time instant
\(t_i \in \mathbb{T}\), choose a parameterized function \(z^i(\cdot;\theta)\) for modeling
\(\tilde{z}\). Then, for any \(y_0 \in \mathbb{R}\), simulate the triple
\((\widetilde{X},\tilde{y},\tilde{z})\) on the time grid \(\mathbb{T}\) by
\begin{equation}
\label{eq:deep-BSDE-scheme}
\left\{ \begin{aligned} \widetilde{X}_0 &:= x_0, \quad \tilde{y}_0 := y_0,\quad \tilde{z}_0 := z^0( \widetilde{X}_0;\theta), \\ \widetilde{X}_{t_{i+1}}
&:= \widetilde{X}_{t_i} + b(t_i,\widetilde{X}_{t_i}, \tilde{y}_{t_i},
\tilde{z}_{t_i})\,\Delta t + \sigma(t_i,\widetilde{X}_{t_i}, \tilde{y}_{t_i},
\tilde{z}_{t_i})\,\Delta W_{t_{i}} ,\\ \tilde{y}_{t_{i+1}} &:= \tilde{y}_{t_i} -
f(t_i,\widetilde{X}_{t_i}, \tilde{y}_{t_i}, \tilde{z}_{t_i})\,\Delta t +
\tilde{z}_{t_i}\,\Delta W_{t_{i}} , \\ \tilde{z}_{t_{i+1}} &:=
z^{i+1}(\widetilde{X}_{t_{i+1}};\theta). \end{aligned} \right.
\end{equation}
The
optimization problem in the deep BSDE method is formulated as
\begin{equation}
\label{eq:deep-BSDE-optimization-problem}
\min_{(y_0,\theta)} \mathbb{E}| \tilde{y}_T - g( \widetilde{X}_T ) |^2.
\end{equation}
Note that this objective value is exactly the BML value of the
simulated \((\tilde{y},\tilde{z})\) multiplied by a constant factor $T$.
Indeed, the residual error
process \(R\) for the simulated triple
\((\widetilde{X},\tilde{y},\tilde{z})\) under scheme \eqref{eq:deep-BSDE-scheme} is
$$ \begin{aligned} R_{t_i} =& \biggl( \tilde{y}_T - \sum_{k=i}^{H-1}
[\tilde{y}_{t_{k+1}} - \tilde{y}_{t_k}] \biggr) -
\biggl(g(\widetilde{X}_{T}) \biggr.\\
&\biggl.+ \sum_{k=i}^{H-1} f(t_k,
\widetilde{X}_{t_k}, \tilde{y}_{t_k}, \tilde{z}_{t_k})\,\Delta t -
\sum_{k=i}^{H-1} \tilde{z}_{t_k}\,(W_{t_{k+1}} - W_{t_k}) \biggr) \\ =&\tilde{y}_T - g( \widetilde{X}_T ), \qquad \text{ for any } 0 \leq i \leq H.
\end{aligned} $$
Therefore,
the parameterization scheme \eqref{eq:deep-BSDE-scheme} leads to
$$\mathbb{E}\int_0^T |R_t|^2\,dt = T\cdot \mathbb{E}| \tilde{y}_T - g( \widetilde{X}_T ) |^2.$$
This shows that our framework \eqref{eq:def-optimization-problem}--\eqref{eq:approximate-optimization-problem} under scheme \eqref{eq:deep-BSDE-scheme} recovers the deep BSDE method
\eqref{eq:deep-BSDE-optimization-problem}.


We also note that, beyond the original formulation in \cite{hanSolvingHighdimensionalPartial2018}, the deep BSDE method has been studied theoretically in \cite{hanConvergenceDeepBSDE2020}.
Although \cite{hanSolvingHighdimensionalPartial2018} initially focuses on decoupled FBSDEs, \cite{hanConvergenceDeepBSDE2020} analyzes the algorithmic error under the fundamental assumption that the forward equation does not involve $Z$.
Their analysis is carried out at the level of a time-discretized scheme and mainly quantifies how the approximation error depends on the time step. This perspective is different from ours. In the present paper, the theory is developed in continuous time and focuses on the interpretation of the objective function itself, rather than on the discretization error induced by a particular time grid.

\begin{table*}[tb]
\centering
\caption{Comparison with relevant optimization-based methods.}
\label{tab:relation-existing-methods}
\renewcommand{\arraystretch}{1.08}
\setlength{\tabcolsep}{4pt}
\begin{tabular}{p{2.4cm} p{2.3cm} p{4.0cm} p{2.0cm} p{5.3cm}}
\toprule
Method & Objective & Problem & Variables & Note \\
\midrule
BML
& Integral residual
& General coupled FBSDEs
& \(\tilde y, \tilde z\)
& Error analysis in continuous time; require contraction property; two-sided error bounds for decoupled case \\
\midrule
Deep BSDE
& Terminal matching
& General coupled FBSDEs
& \(y_0,\tilde z\)
& Error analysis after discretization; require $Z$ absence in forward equation\\
\midrule
Martingale
& Integral value residual
& Linear decoupled FBSDEs
& \(\tilde y\)
& Error analysis in continuous time; exact \(Y\)-error identity in its linear setting \\
\midrule
Deep BSDE-ML
& Terminal matching with fixed \(Y_0\)
& Linear decoupled FBSDEs
& \(\tilde z\)
& Error analysis in continuous time; exact \(Z\)-error identity in its linear setting \\
\midrule
Robust deep FBSDE
& Terminal matching with fixed \(Y_0\) and variance penalty
& Control-induced coupled FBSDEs
& \(\tilde z\)
& Error analysis after discretization; exploit the stochastic control interpretation\\
\bottomrule
\end{tabular}
\end{table*}

\textbf{Relation to the martingale loss
\cite{jiaPolicyEvaluationTemporaldifference2022}.} That method is
developed for the policy evaluation problem in continuous-time
reinforcement learning, where the considered FBSDE is decoupled and
essentially reduces to the following BSDE
\begin{equation}
\label{eq:simple-BSDE}
Y_t = \xi + \int_t^T r_s\,ds -
\int_t^T Z_s\,dW_s, \quad t\in[0,T].
\end{equation}
The optimization problem formulated in
\cite{jiaPolicyEvaluationTemporaldifference2022} is to minimize the
\emph{martingale loss} $$ \min_\theta \frac{1}{2} \mathbb{E}\int_0^T
\biggl| \xi + \int_t^T r_s\,ds - J^\theta_t \biggr|^2\,dt, $$
where $J^\theta$ is parameterized process to model \(\tilde{y}\).
This objective is exactly the BML value of the trial solution $(J^\theta,0)$ multiplied by a
constant factor $1/2$.
Indeed,
the residual error process \(R\) for the $(J^\theta,0)$ is
$R_t = J^\theta_t - \xi - \int_t^Tr_s\,ds$,
showing that
$$
\mathbb{E}\int_0^T |R_t|^2\,dt = \mathbb{E}\int_0^T
\biggl| \xi + \int_t^T r_s\,ds - J^\theta_t \biggr|^2\,dt.
$$
In \cite{jiaPolicyEvaluationTemporaldifference2022},
the objective is already given in integral form and is proved to be exactly the mean squared error of the \(Y\)-component, i.e.,
the error between the trial value process and the true one.
This interpretation is close in spirit to ours,
but it characterizes only the backward process $Y$.
The reason is that the policy evaluation problem considered there reduces to a linear decoupled BSDE,
and the method is formulated directly at the level of the value process.
From the BML viewpoint of the present paper,
this corresponds to a restricted trial solution in which the \(Z\)-component is fixed as $\tilde z=0$.

\textbf{Relation to the deep BSDE-ML method
\cite{wangDeepBSDEMLLearning2022}.}
That work studies the same policy evaluation problem as the martingale-loss method,
namely the linear decoupled BSDE \eqref{eq:simple-BSDE}.
In this sense,
it is almost parallel to \cite{jiaPolicyEvaluationTemporaldifference2022}.
While the martingale-loss method is formulated to characterize the \(Y\)-component,
deep BSDE-ML is formulated to characterize the \(Z\)-component.
More precisely, \cite{wangDeepBSDEMLLearning2022} proves that its proposed loss is exactly the mean squared error between the trial diffusion term $\tilde{z}$ and the true solution $Z$, giving a mean-squared interpretation analogous to that of the martingale loss.

Its objective is also closely related to the deep BSDE method. For \eqref{eq:simple-BSDE},
the deep BSDE-ML loss effectively coincides with the terminal-matching objective of deep BSDE,
except that the initial value is not treated as a free decision variable.
Instead, \(Y_0\) is fixed to the expectation of the cumulative reward $\mathbb{E}\bigl[\xi+\int_0^T r_t\,dt\bigr]$.
From the viewpoint of the present paper,
deep BSDE-ML can therefore be regarded as a \(Z\)-only counterpart of the martingale-loss method within the same policy evaluation setting.

\textbf{Relation to the robust deep FBSDE method
\cite{anderssonConvergenceRobustDeep2023}.}
That work considers strongly coupled FBSDEs arising from stochastic control and develops a modified deep-learning-based solver for this problem class.
Similar to the deep BSDE-ML viewpoint,
the method parameterizes the \(Z\)-component and treat the initial value of the backward equation as a term depending on the \(Z\)-component.
It exploits the stochastic control interpretation of the FBSDE and introduces a new loss consisting of two parts: a control-cost term, corresponding to the mean of the stochastic cost, and a variance penalty term, which coincides with the mean squared terminal mismatch.

In \cite{anderssonConvergenceRobustDeep2023},
the error analysis focuses on the discretized objective and the convergence of the associated numerical approximation.
By contrast, our framework starts from a continuous-time integral-form objective for a general trial pair \((\tilde y,\tilde z)\), and the main theory concerns the interpretation of that objective itself through the Picard residual. From this perspective, \cite{anderssonConvergenceRobustDeep2023} may be viewed as a problem-specific robust objective design for control-induced coupled FBSDEs, whereas our contribution is a continuous-time residual-based formulation to general coupled FBSDEs.

\textbf{Distinctions and Highlights.}
Table~\ref{tab:relation-existing-methods} summarizes the main differences between the proposed BML framework and relevant optimization-based methods.
The novelty of the present work is a framework-level interpretation:
we formulate a continuous-time residual objective, connect it exactly to the Picard operator, and derive convergence and error statements in terms of the same objective minimized numerically.
Moreover, the framework is not restricted to optimizing only \(Y\), only \(Z\), or only the initial value \(Y_0\), and the same interpretation applies to general coupled FBSDEs.

From a practical viewpoint, our parallel parameterization treats \(y^\theta\) and \(z^\theta\) symmetrically and evaluates them directly at sampled space-time points once the forward path has been generated.
This differs from the usual deep BSDE recursion, which is sequential along the time grid, and may offer potential benefits for parallel computation.

It is worth noting that existing error analyses such as \cite{hanConvergenceDeepBSDE2020} for the deep BSDE method, and \cite{anderssonConvergenceRobustDeep2023} for the robust deep FBSDE method, are carried out for discretized solvers.
By contrast,
the error analysis in the present paper is based on the objective value being minimized, namely the BML value, which is defined in integral form in continuous time.
Therefore, the time-discretization error of the numerical solver is not covered by the current theory.
This is a limitation of the present work and will be an important topic for future research.

\subsection{Analytical Examples}
We now provide two examples and evaluate their BML values to demonstrate how
the proposed framework works. Note that this subsection evaluates BML
values analytically without taking time discretization. The considered
examples will be examined numerically in the next section.

\begin{example}[A Toy BSDE]
\label{example:a-toy-BSDE}
Let $W$ be a \(d\)-dimensional standard Brownian motion. Consider the
following toy BSDE
\cite{jiaPolicyEvaluationTemporaldifference2022,wangProbabilisticFrameworkHowards2023},
\begin{equation}
\label{eq:a-toy-BSDE}
Y_t = \frac{ |W_T|^2}{d} - \int_t^T\,ds - \int_t^T Z_s^\intercal \,dW_s, \quad t\in[0,T].
\end{equation}
Consider the parameterization scheme \(\tilde{y}_t = \theta_1 |W_t|^2\) and
\(\tilde{z}_t = \theta_2 W_t\) for \(\theta_1,\theta_2\in \mathbb{R}\).
\end{example}
To evaluate the BML value for BSDE~\eqref{eq:a-toy-BSDE} corresponding
to \((\theta_1,\theta_2)\), we rewrite the residual error
process
$$ \begin{aligned}
R_t =& \theta_1|W_t|^2 - \frac{|W_T|^2}{d}
  +(T-t) + \theta_2 \int_t^T W_s^\intercal \,dW_s \\
  =& \Bigl(\theta_1 - \frac{1}{2}\theta_2 \Bigr)|W_t|^2 +
  \Bigl(\frac{1}{2}\theta_2 - \frac{1}{d} \Bigr) |W_T|^2 \\
  &-
  \Bigl(\frac{d}{2}\theta_2 - 1\Bigr) (T-t) . \end{aligned} $$
Therefore, the BML value \(\mathbb{E}\int_0^T |R_t|^2\,dt\) is a
quadratic function of \(\theta_1\) and \(\theta_2\), and has a global
minimizer \(\theta_1^*=\frac{1}{2}\theta_2^*=\frac{1}{d}\). It can be verified by Itô's
formula that the trial solution \((\tilde{y},\tilde{z})\) with optimal
parameters is indeed the true solution \((Y^*,Z^*)\).

A direct calculation shows that
\begin{equation}
\label{eq:theoretical-BML-toy-BSDE-parameterization-scheme-1}
\operatorname{BML}(\theta_1,\theta_2) =
\frac{T^3}{3}(d+2)d\Bigl(\theta_1 - \frac{1}{d}\Bigr)^2 +
\frac{T^3}{3}d \Bigl( \theta_2 - \frac{2}{d} \Bigr)^2,
\end{equation}
which can
also be obtained from \(\|(\tilde{y}-Y^*,\tilde{z}-Z^*)\|_\mu\) as
suggested by
Theorem~\ref{theorem:the-Picard-interpretations-of-BML-values}.
With a slight abuse of notation,
we use $\operatorname{BML}(\theta_1,\theta_2)$ to denote the BML value of the parameterized processes $\{\tilde{y},\tilde{z}\}$ corresponding to the given parameters $(\theta_1,\theta_2)$.

In general, the BML value for a linear BSDE is quadratic in the
parameter vector $\theta$ when both $\tilde{y}$ and $\tilde{z}$ are
linear in $\theta$.

\begin{example}[A coupled FBSDE]
\label{example:a-coupled-FBSDE}
Let $W$ be a \(d\)-dimensional standard Brownian motion. Consider the following coupled FBSDE
\cite{benderTimeDiscretizationMarkovian2008,hanConvergenceDeepBSDE2020} \begin{equation}
\label{eq:a-coupled-FBSDE}
\left\{\begin{aligned}
X_t= & x_0 +
\int_0^t \sigma_0 Y_s\,dW_s, \\
Y_t= &\int_t^T\left[-r Y_s+\frac{ \sigma_0^2}{2} e^{-3 r(T-s)}\Bigl( A \sum_{j=1}^d \sin
\left(X_{j, s}\right)\Bigr)^3\right]\,ds\\
&+A \sum_{j=1}^d \sin \left(X_{j, T}\right)-\int_t^TZ_s^\intercal\,d W_s,
\end{aligned}\right.
\end{equation}
where \(A,\sigma_0,r\) are constants and \(X_{j,s}\)
refers to the \(j\)-th component of \(X_s\).

Consider the parameterization scheme $$ \left\{ \begin{aligned}
\widetilde{X}_{j,t} &= x_0 + \int_0^t \sigma_0 \tilde{y}_t \,dW_{j,s},
\\ \tilde{y}_t &= \theta_1 e^{-r(T-t)}\sum_{j'=1}^d \sin
\left(\widetilde{X}_{j', t}\right),\\ \tilde{z}_{j,t} &= \theta_2
e^{-2r(T-t)}\left(\sum_{j'=1}^d \sin (\widetilde{X}_{j', t})\right)\cos
(\widetilde{X}_{j,t}) \end{aligned} \right. $$ for \(\theta_1,\theta_2\in \mathbb{R}\).
\end{example}
To evaluate the BML value for FBSDE~\eqref{eq:a-coupled-FBSDE}
corresponding to \((\theta_1,\theta_2)\), rewrite the residual error
process $$ \begin{aligned}
R_t =& \tilde{y}_t - A \sum_{j=1}^d \sin
  \left(\widetilde{X}_{j, T}\right)+\int_t^T \tilde{z}_s^\intercal\,d   W_s\\
  &-\int_t^T\left[-r
    \tilde{y}_s+\frac{ \sigma_0^2}{2} e^{-3 r(T-s)}\Bigl( A
    \sum_{j=1}^d \sin \left(\widetilde{X}_{j,
      s}\right)\Bigr)^3\right]\,ds \\
      =& \theta_1 \sum_{j=1}^d \sin \left(\widetilde{X}_{j,
    T}\right)-A \sum_{j=1}^d \sin \left(\widetilde{X}_{j,
    T}\right) \\
    &+\int_t^T\left[-r \tilde{y}_s+\frac{ \sigma_0^2}{2}
    e^{-3 r(T-s)}\Bigl( \theta_1 \sum_{j=1}^d \sin
    \left(\widetilde{X}_{j, s}\right)\Bigr)^3\right]\,ds \\
    &-\int_t^T\left[-r \tilde{y}_s+\frac{ \sigma_0^2}{2} e^{-3
      r(T-s)}\Bigl( A \sum_{j=1}^d \sin \left(\widetilde{X}_{j,
      s}\right)\Bigr)^3\right]\,ds \\
      &+(\theta_2 -
  \sigma_0\theta_1^2)\int_t^T e^{-2r(T-s)} \sum_{j=1}^d \sin
  \left(\widetilde{X}_{j, s}\right) \langle \cos \widetilde{X}_s
  ,\,dW_s \rangle \\
  =& (\theta_1^3 - A^3)\int_t^T \frac{
    \sigma_0^2}{2} e^{-3 r(T-s)}\Bigl( \sum_{j=1}^d \sin
  \left(\widetilde{X}_{j, s}\right)\Bigr)^3\,ds \\
  &+(\theta_2 -
  \sigma_0\theta_1^2)\int_t^T e^{-2r(T-s)} \sum_{j=1}^d \sin
  \left(\widetilde{X}_{j, s}\right) \langle \cos \widetilde{X}_s
  ,\,dW_s \rangle\\
  &+(\theta_1 - A)\sum_{j=1}^d \sin
  \left(\widetilde{X}_{j, T}\right). \end{aligned} $$
Therefore, the BML value \(\mathbb{E}\int_0^T |R_t|^2\,dt\) has a
global minimizer \(\theta_1^* = A\) and \(\theta_2^*=\sigma_0
A^2\). It can be verified by Itô's formula that the trial solution
\((\tilde{y},\tilde{z})\) with optimal parameters is indeed the true
solution \((Y^*,Z^*)\).

The parameterization schemes discussed above are not practical as they
rely on prior knowledge of the solution form. They are provided here
for illustration purposes. In practice, the trial solution
\((\tilde{y},\tilde{z})\) is parameterized by generic function
approximators, e.g., neural networks. In the next section, we will revisit
these examples under practical parameterization schemes and provide
numerical results obtained via gradient-based optimization methods.

\section{Numerical Examples}
\label{section:experiments}
In this section, we numerically revisit the examples described in the
previous section and consider an additional high-dimensional example derived from
a nonlinear stochastic optimal control problem. The purpose of these experiments is to validate the proposed BML objective and its theoretical interpretation, including the relation between small BML values and small solution errors in the tested problems. They are not intended as a comprehensive benchmark against existing numerical solvers.

\subsection{Visualize BML Values}
After choosing a parameterization scheme for trial solutions, BML can
be visualized as a finite-dimensional function.

\begin{example}[A Toy BSDE---Revision 1]
\label{example:a-toy-BSDE-revision-1}
Set Example~\ref{example:a-toy-BSDE} with \(T=1\) and \(d=3\).
Visualize empirical BML value
\eqref{eq:approximate-optimization-problem} and theoretical BML value
\eqref{eq:theoretical-BML-toy-BSDE-parameterization-scheme-1} by
varying \((\theta_1,\theta_2)\). According to
Example~\ref{example:a-toy-BSDE}, the optimal parameters are
$\theta^*_1= \frac{1}{d},~\theta^*_2 = \frac{2}{d}$.
\end{example}

For each \((\theta_1,\theta_2)\), we estimate empirical BML
\eqref{eq:approximate-optimization-problem} for
BSDE~\eqref{eq:a-toy-BSDE} using \(10^6\) Monte Carlo samples and
\(10^3\) time intervals with Euler-Maruyama method. Results are
presented in Figure~\ref{fig:visualize-BML-toy-BSDE}.

\begin{example}[A Coupled FBSDE---Revision 1]
\label{example:a-coupled-FBSDE-revision-1}
Set Example~\ref{example:a-coupled-FBSDE} with \(T=1\) and \(d=3\).
Set additional problem-specific parameters to
\(A=1,~\sigma_0=0.3,~r=0.1\) and $x_0=(
\frac{\pi}{2},\frac{\pi}{2},\frac{\pi}{2})$. Visualize empirical BML
value \eqref{eq:approximate-optimization-problem} by varying
\((\theta_1,\theta_2)\). According to
Example~\ref{example:a-coupled-FBSDE}, the optimal parameters are
$\theta^*_1= A,~\theta^*_2 = \sigma_0A^2$.
\end{example}

For each \((\theta_1,\theta_2)\), we estimate empirical BML
\eqref{eq:approximate-optimization-problem} for
FBSDE~\eqref{eq:a-coupled-FBSDE} using \(10^6\) Monte Carlo samples
and \(10^3\) time intervals with Euler-Maruyama method. Results are
presented in Figure~\ref{fig:visualize-BML-coupled-FBSDE}.

\subsection{Optimize BML Values}
With stochastic gradient descent, the BML value can be optimized with
a relatively small number of samples at each iteration.

To avoid requiring prior knowledge of the true solution form, this
subsection demonstrates different parameterization schemes for
optimization. The optimization algorithm is chosen as Adam, a popular
variant of SGD \cite{kingmaAdamMethodStochastic2015}.  Considering
the stochastic nature of this algorithm, we execute it multiple times
and report metrics during optimization via their mean and standard
error across different runs.

\begin{example}[A Toy BSDE---Revision 2]
\label{example:a-toy-BSDE-revision-2}
Optimize the empirical BML value in
Example~\ref{example:a-toy-BSDE-revision-1} under the parameterization
scheme $$ \tilde{y}_t=\theta_1 |W_t|^4, \quad \tilde{z}_t =
\theta_2|W_t|^2W_t. $$ As a benchmark, the theoretical BML value under
this parameterization scheme is
\begin{equation}
\label{eq:theoretical-BML-toy-BSDE-parameterization-scheme-2}
\operatorname{BML}(\theta_1,\theta_2)= C_1(\theta_1 - \theta_1^*)^2 + C_2(\theta_2 - \theta_2^*)^2 + C_3,
\end{equation}
where \(C_1,C_2,C_3\) are positive constants and \(\theta_1^*=
\frac{5}{4d(d+6)T}\approx0.0463,~\theta_2^*=\frac{5}{2d(d+4)T}\approx0.119\). This theoretical
expression is calculated via
Theorem~\ref{theorem:the-Picard-interpretations-of-BML-values}. The
theoretical minimum is $C_3\approx0.297$.
\end{example}

This example suggests that the optimal \(\tilde{z}\) may not always
coincide with the diffusion term of \(d\tilde{y}_t\). In this
example, differentiating \(\tilde{y}_t\) yields a diffusion term
\(4\theta_1|W_t|^2 W_t\). However, the optimal parameter for
\(\tilde{z}\) is \(\theta_2^*\neq 4 \theta_1^*\). By
Theorem~\ref{theorem:the-Picard-interpretations-of-BML-values}, the
optimal \(\tilde{y}\) minimizes \(\mathbb{E}\int_0^T| \tilde{y}_t -
\frac{1}{d}|W_t|^2|^2\,dt\), while the optimal \(\tilde{z}\) minimizes
\(\mathbb{E}\int_0^T| \tilde{z}_t -
\frac{2}{d}W_t|^2\,dt\). Minimizers of these problems heavily depend
on their parameterization schemes.

At each optimization step, we estimate empirical BML
\eqref{eq:approximate-optimization-problem} using $10^3$ Monte Carlo
samples and $10^3$ time intervals.  Learning rates set to $10^{-3}$
for $\theta_1$ and $3 \times 10^{-3}$ for $\theta_2$.  Meanwhile, we
plot the parameter values $(\theta_1,\theta_2)$ during the
optimization. Results are presented in
Figure~\ref{fig:optimize-BML-toy-BSDE}.

\begin{example}[A Coupled FBSDE---Revision 2]
\label{example:a-coupled-FBSDE-revision-2}
Optimize the empirical BML value in
Example~\ref{example:a-coupled-FBSDE-revision-1} under the
parameterization scheme $$ \left\{ \begin{aligned} \widetilde{X}_{j,t}
  &= x_0 + \int_0^t \sigma_0 y^\theta(s,\widetilde{X}_{s}) \,dW_{j,s},
  \\ \tilde{y}_t &= y^\theta(s,\widetilde{X}_{s}), \qquad \tilde{z}_t
  = z^\theta(s,\widetilde{X}_{s}). \end{aligned} \right. $$ Here,
parameterized functions \(y^\theta\) and \(z^\theta\) are neural
networks. As a benchmark, the true solution $(Y^*,Z^*)$ can be
simulated via choosing optimal $(\theta_1^*,\theta_2^*)$ in
Example~\ref{example:a-coupled-FBSDE-revision-1}.
\end{example}

We construct \(y^\theta\) and \(z^\theta\) as functions of $(t, x) \in
[0,T] \times \mathbb{R}^n$. First, initialize three one-hidden-layer
ReLU neural networks: $\phi_t^\theta: [0,T] \to
\mathbb{R}^{n_t},~\phi_y^\theta: \mathbb{R}^{n_t+n} \to
\mathbb{R}^{m},~\phi_z^\theta:\mathbb{R}^{n_t+n} \to
\mathbb{R}^{d}$. Then, for any $(t,x)$, set
$$ y^\theta(t,x):= \phi_y^\theta(\phi_t^\theta(t), x), \quad
z^\theta(t,x):=\phi_z^\theta(\phi_t^\theta(t),x). $$ The hidden sizes
of these one-hidden-layer networks are 4 for $\phi_t$, 32 for
$\phi_y$, and 32 for $\phi_z$. Finally, the embedding dimension $n_t$
is set to 4.

At each optimization step, we estimate empirical BML
\eqref{eq:approximate-optimization-problem} using $10^3$ Monte Carlo
samples and $20$ time intervals. The learning rate is set to $10^{-3}$ for all parameters.  Meanwhile, we estimate errors between the
trial solution $(\tilde{y}, \tilde{z})$ and the true solution
$(Y^*,Z^*)$ under different norms. Results are presented in
Figure~\ref{fig:optimize-BML-coupled-FBSDE}.

\begin{figure}[h]
  \centering
  \includegraphics[width=.45\textwidth]{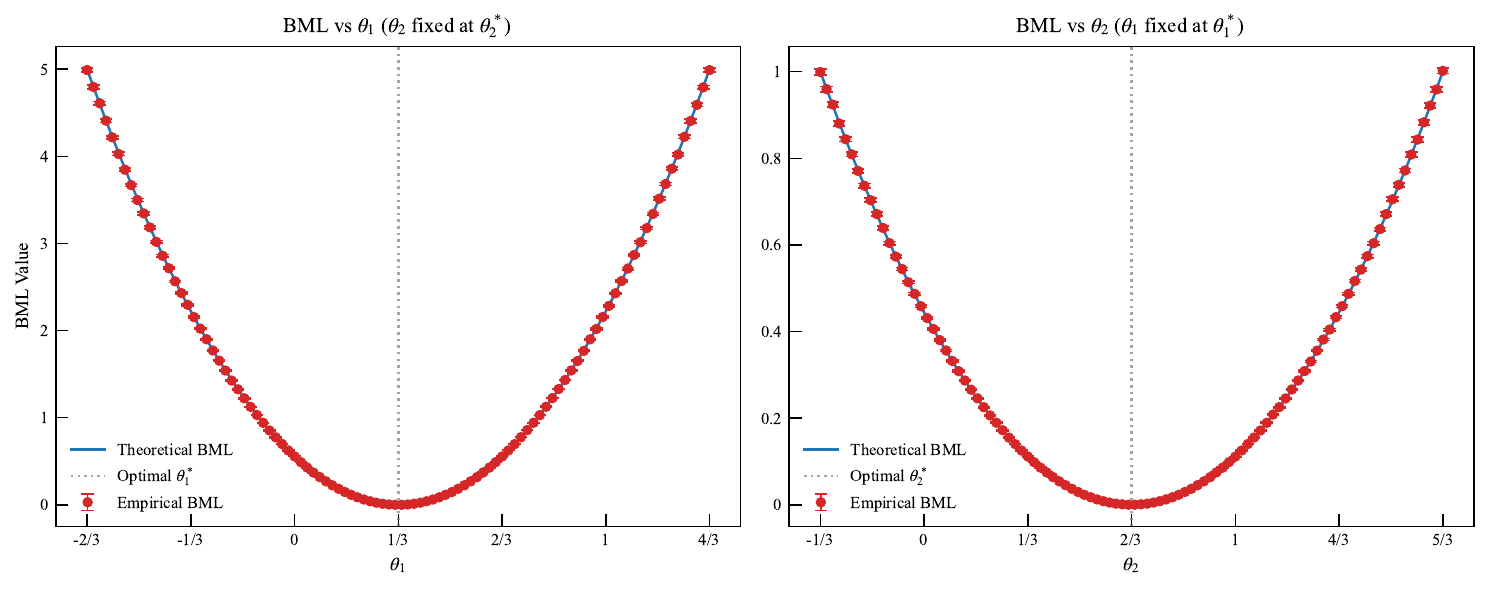}
  \caption{\label{fig:visualize-BML-toy-BSDE}Visualization of the BML value for
    Example~\ref{example:a-toy-BSDE-revision-1}. Left: $\theta_1 \in
    [\theta_1^*-1,\theta_1^*+1]$ with $\theta_2=\theta_2^*$. Right:
    $\theta_2\in[\theta_2^*-1,\theta_2^*+1]$ with
    $\theta_1=\theta_1^*$. Error bars indicate 99.7\% confidence intervals of empirical expectations.
  }
\end{figure}

\begin{figure}[h]
  \centering
  \includegraphics[width=.45\textwidth]{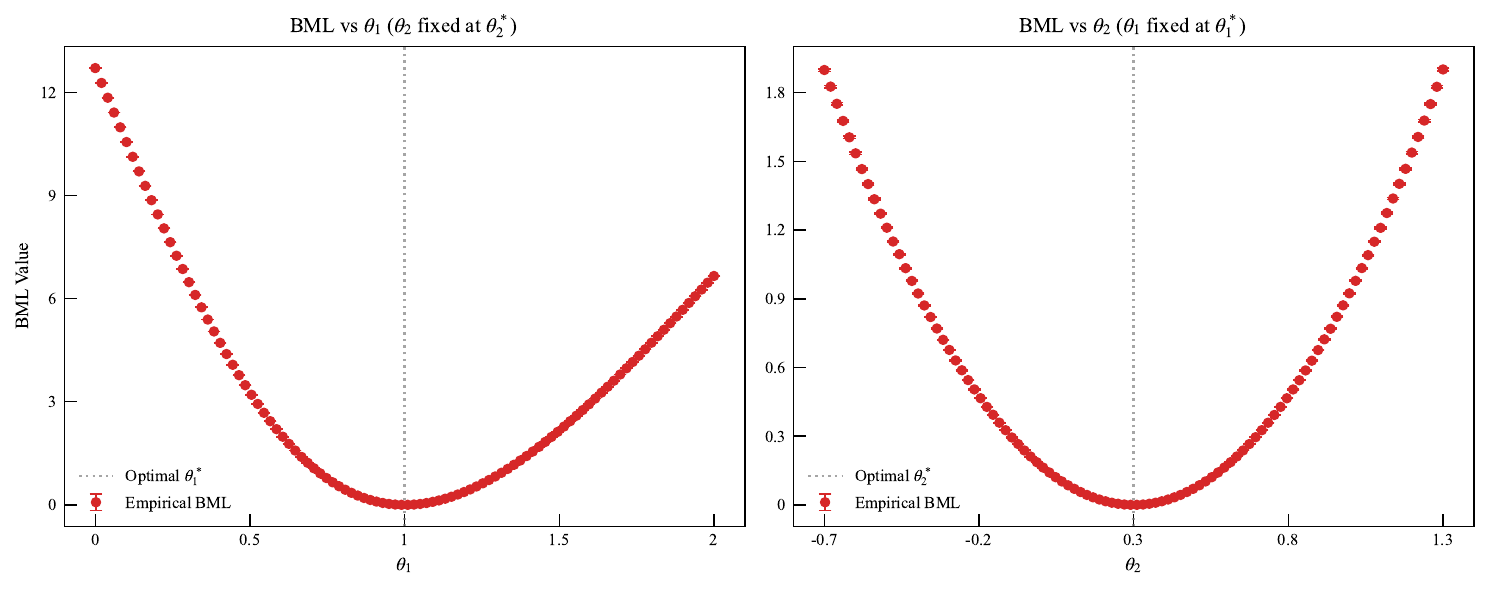}
  \caption{\label{fig:visualize-BML-coupled-FBSDE}Visualization of the BML value for
    Example~\ref{example:a-coupled-FBSDE-revision-1}. Left: $\theta_1
    \in [\theta_1^*-1,\theta_1^*+1]$ with
    $\theta_2=\theta_2^*$. Right:
    $\theta_2\in[\theta_2^*-1,\theta_2^*+1]$ with
    $\theta_1=\theta_1^*$. Error bars indicate 99.7\% confidence
    intervals of empirical expectations.  }
\end{figure}

\begin{figure}[h]
  \centering
  \includegraphics[width=.45\textwidth]{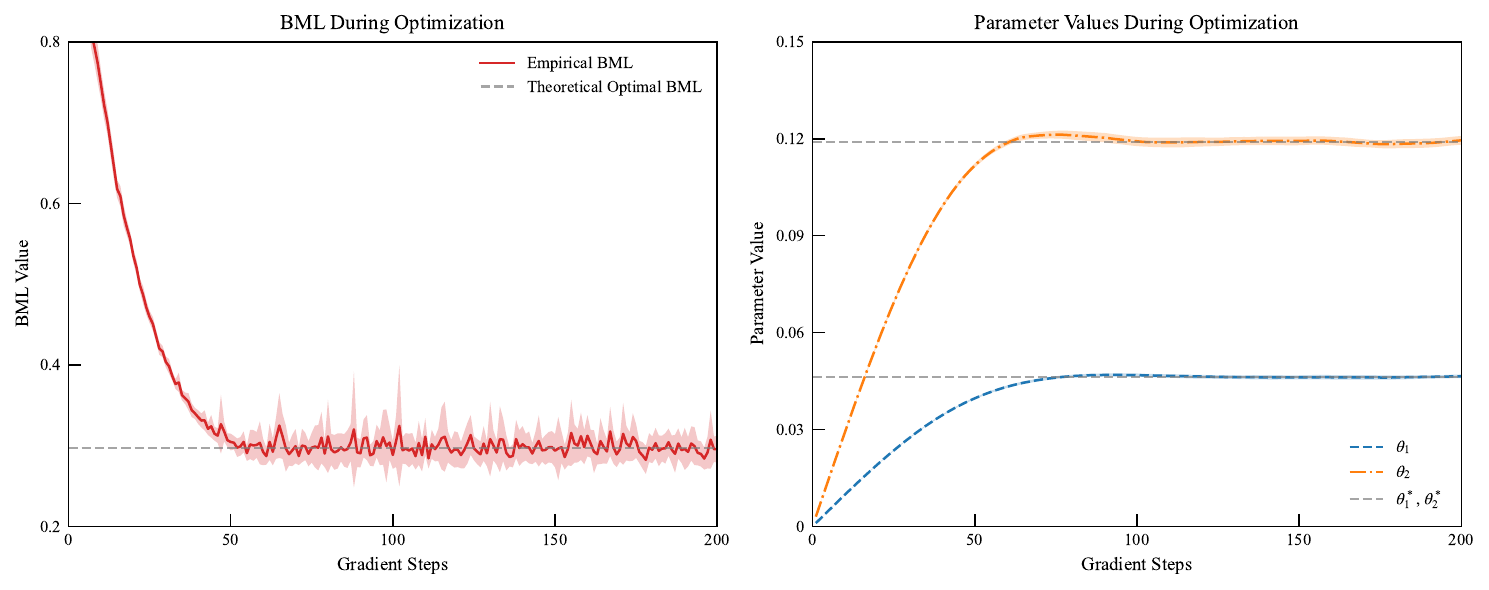}
  \caption{\label{fig:optimize-BML-toy-BSDE}Optimization of the empirical
    BML value for Example~\ref{example:a-toy-BSDE-revision-2}. Left:
    empirical BML values and its theoretical optimum. Right: parameter
    values $\theta_1,\theta_2$, and their theoretical optimum.  Metrics
    averaged over 50 independent runs; shaded regions indicate $\pm 3$
    standard errors.}
\end{figure}

\subsection{A 1000 Dimensional HJB Equation}
The proposed framework can be applied to solve
Hamilton--Jacobi--Bellman equations.

\begin{example}[A 1000D HJB Equation]
\label{example:a-HJB-decoupled-FBSDE}
Consider the following stochastic optimal control problem in
\(n\)-dimensions
\cite{hanSolvingHighdimensionalPartial2018,huTacklingCurseDimensionality2024}
$$ \begin{aligned} \min&\quad \mathbb{E}\biggl[g(x_T) + \int_0^T
    \|u_t\|^2\,dt\biggr] \\ \operatorname{s.t.}&\quad x_t = x_0 +
  2\sqrt{\lambda} \int_0^t u_s\,ds + \sqrt{2}\,W_t, \end{aligned} $$
where \(\lambda\) is a given positive constant, $\{x_t\}$ and
$\{u_t\}$ are processes valued in $\mathbb{R}^n$. The associated HJB
equation is $$ \partial_t v + \Delta v - \lambda \|\nabla v\|^2 = 0,
\quad v(T, \cdot) = g(\cdot). $$ By the nonlinear Feynman-Kac formula,
the value function $v$ is related to the solution $(X^*,Y^*,Z^*)$ of
the following FBSDE
\begin{equation}
\label{eq:HJB-decoupled-FBSDE}
\left\{ \begin{aligned}
X_t &= x_0 + \sqrt{2}W_t, \\
Y_t &= g(X_T) + \int_t^T -\frac{\lambda}{2} |Z_s|^2\,ds - \int_t^T Z_s^\intercal\,dW_s
\end{aligned} \right.
\end{equation}
via $Y^*_t=v(t,X_t^*)$. In particular, the optimal cost
$v(0,x_0)=Y^*_0$.

Set $T=1,~\lambda=1,~x_0=(0,0,\ldots,0)$, and the terminal
condition $g(x):=\ln((1+|x|^2)/2)$. As a benchmark, the optimal cost
$Y^*_0$ can be obtained by applying Hopf-Cole transformation to the
HJB equation
\begin{equation}
\label{eq:benchmark-Y0-HJB-decoupled-FBSDE}
Y^*_0 = v(0,x_0) = - \frac{1}{\lambda} \ln \biggl(
\mathbb{E}\Bigl[ \exp\bigl(- \lambda g(x_0 +
  \sqrt{2}W_T)\bigr)\Bigr]\biggr).
\end{equation}
\end{example}

To solve FBSDE~\eqref{eq:HJB-decoupled-FBSDE} and obtain the optimal
cost, we optimize the empirical BML value under the parameterization
scheme
$$ \left\{ \begin{aligned} \widetilde{X}_{t} &= x_0 + \sqrt{2}W_{t},
  \\ \tilde{y}_t &= y^\theta(s,\widetilde{X}_{s}), \qquad \tilde{z}_t
  = z^\theta(s,\widetilde{X}_{s}). \end{aligned} \right. $$ The
construction of $y^\theta$ and $z^\theta$ follows
Example~\ref{example:a-coupled-FBSDE-revision-2}.

At each optimization step, we estimate empirical BML
\eqref{eq:approximate-optimization-problem} using $10^3$ Monte Carlo
samples and $20$ time intervals. The learning rate is set to $10^{-3}$
for all parameters. Meanwhile, we plot the prediction $\tilde{y}_0$
during the optimization. To demonstrate the capability of the method in
high dimensions, we solve the problem for
$n\in\{100,250,500,1000\}$. Results are presented in
Figure~\ref{fig:optimize-BML-HJB-decoupled-FBSDE}, with final
predictions and the relative errors reported in
Table~\ref{table:Y0-error-HJB-decoupled-FBSDE}.

At first glance, the decreasing relative errors with increasing
dimension in Table~\ref{table:Y0-error-HJB-decoupled-FBSDE} may appear
anomalous. This phenomenon emerges because higher-dimensional networks
possess greater approximation capacity---since the parameter count in
a one-hidden-layer network scales linearly with the output dimension
$n$. The diminishing BML values directly reflect this enhanced
representational power, confirming the effectiveness of our
minimization approach.
This counterintuitive phenomenon is consistent with the proposed framework:
in this experiment, smaller optimized BML values are accompanied by smaller relative errors.

\begin{figure}[h]
  \centering
  \includegraphics[width=.45\textwidth]{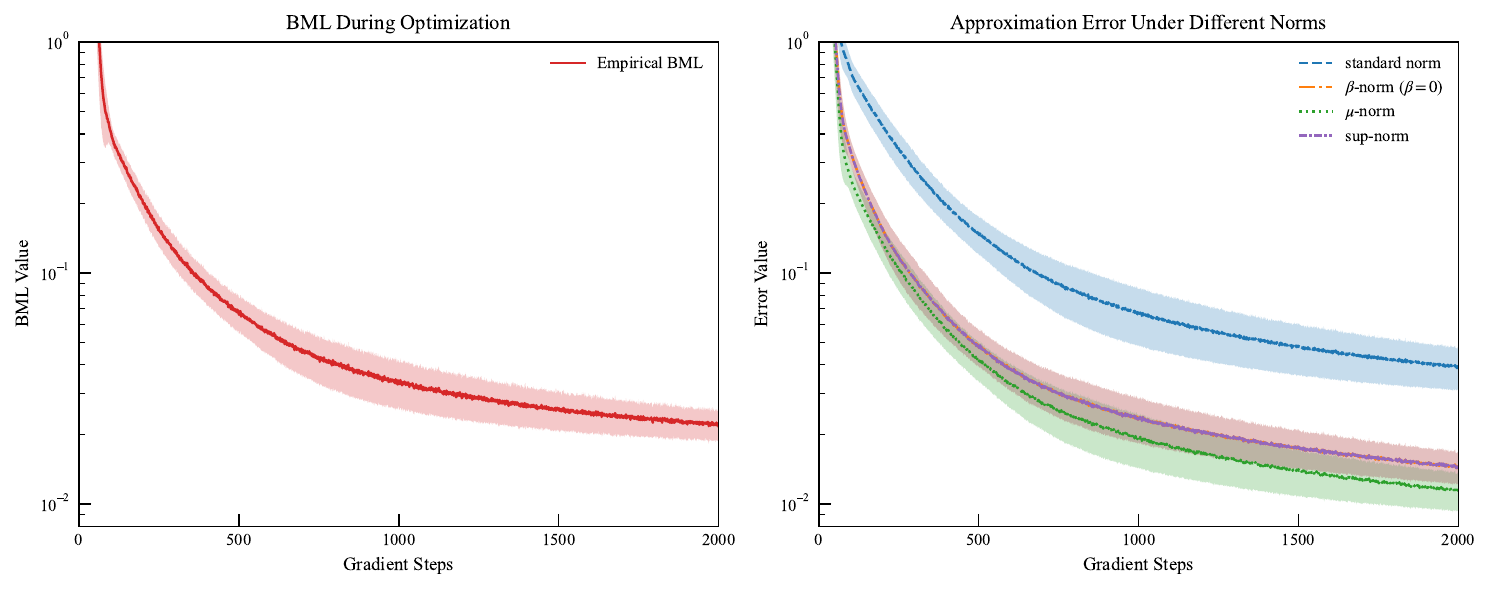}
  \caption{\label{fig:optimize-BML-coupled-FBSDE}Optimization of the
    empirical BML value for
    Example~\ref{example:a-coupled-FBSDE-revision-2}. Left: empirical
    BML values. Right: approximation error between $(\tilde{y},
    \tilde{z})$ and $(Y^*,Z^*)$ under different norms. In the right
    panel, line $\|(\tilde{y}, \tilde{z})-(Y^*,Z^*)\|_{\beta}^2$
    nearly coincides with line $\|(\tilde{y},
    \tilde{z})-(Y^*,Z^*)\|_\textrm{sup}^2$. Metrics averaged over 50
    independent runs; shaded regions indicate $\pm 3$ standard
    errors.}
\end{figure}

\begin{figure}[h]
  \centering
  \includegraphics[width=.45\textwidth]{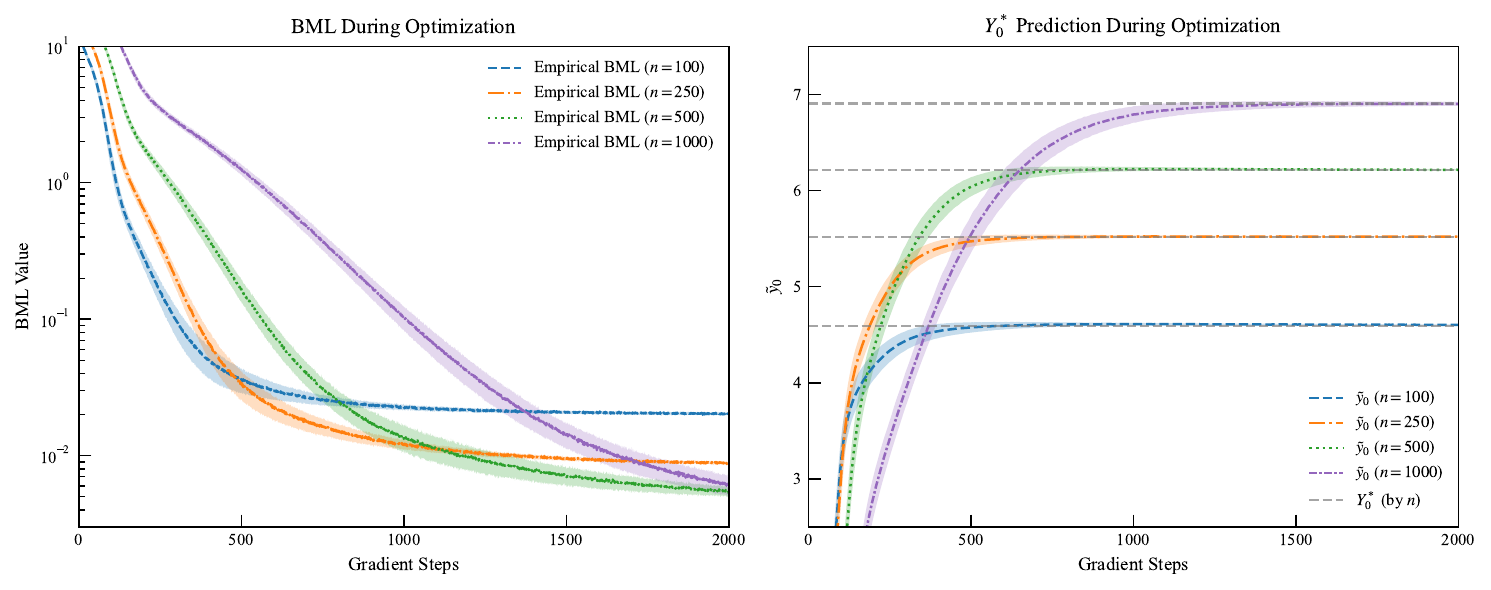}
  \caption{\label{fig:optimize-BML-HJB-decoupled-FBSDE}Optimization of the
    empirical BML value for
    Example~\ref{example:a-HJB-decoupled-FBSDE}. Left: empirical BML
    values. Right: predictions $\tilde{y}_0$. Metrics averaged over 50
    independent runs; shaded regions indicate $\pm 3$ standard
    errors.}
\end{figure}

\begin{table}[h]
\centering
\caption{Final results for
  Example~\ref{example:a-HJB-decoupled-FBSDE}. Benchmark values $Y_0^*$ are
  obtained by the analytical
  expression~\eqref{eq:benchmark-Y0-HJB-decoupled-FBSDE} using $10^9$
  Monte Carlo samples.}
\label{table:Y0-error-HJB-decoupled-FBSDE}
\begin{tabular}{ccccc}
\toprule
Dimension & Final BML & Optimal cost & Prediction & Rel. error \\
($n$) & value & ($Y_0^*$) & ($\tilde{y}_0$) & ($|\tilde{y}_0 - Y_0^*|/Y_0^*$) \\
\midrule
100 & 0.020 & 4.590 & 4.604 & 0.30\% \\
250 & 0.009 & 5.515 & 5.520 & 0.09\% \\
500 & 0.005 & 6.212 & 6.215 & 0.05\% \\
1000 & 0.006 & 6.906 & 6.901 & 0.07\% \\
\bottomrule
\end{tabular}
\end{table}

\section{Conclusion and Future Directions}
\label{section:conclusion}
Instead of explicitly performing the Picard iteration step-by-step, we
propose to find its fixed point directly by minimizing the residual
error of the fixed point equation. For any pair of processes in the
solution space, we use an integral-form value to quantify \emph{how
well} it fits the FBSDE. This value, though defined solely on the
trial solution, is shown equal to the residual error of the fixed
point equation under a particular norm. This result suggests that
minimizing this value could yield the fixed point of the Picard
operator. Relevant convergence results and error bounds are developed
accordingly.


We also positioned the proposed framework relative to several
representative optimization-based methods in the literature. In
particular, we showed that the BML viewpoint provides a common
objective-level interpretation of methods such as deep BSDE,
martingale-loss, Deep BSDE-ML, while differing
from them in scope and emphasis. Our focus is not on a new
problem-specific architecture or a discretization-specific solver, but
on a continuous-time residual formulation for a general trial pair
\((\tilde y,\tilde z)\), together with convergence and error
statements stated in terms of the same objective being minimized.
The numerical examples in this paper are intended to support this
objective-based viewpoint by validating the proposed loss design and
its theoretical interpretation.

Future directions include refining the error bounds under more
sophisticated conditions for coupled FBSDEs, exploring other choices of the
\(\mu\)-measure other than the Lebesgue measure, considering the time-discretization error introduced in estimating the objective function, and carrying out broader empirical comparisons with existing numerical solvers under carefully matched implementations.


\section{Appendix}

\subsection{Existence Results for BSDEs and FBSDEs}
\label{appendix:existence-results}
For completeness, we recall here two standard existence results based
on contraction mapping, together with related remarks.

\begin{lemma}[BSDE Existence]
\label{lemma:BSDE-Existence}
Let Assumption~\ref{assumption:standing-assumption-for-BSDEs}
hold. Then, a pair \((Y, Z) \in \mathcal{M}[0,T]\) is an adapted
solution of BSDE~\eqref{eq:BSDE} if and only if it is a fixed point of
the Picard operator defined by
\eqref{eq:Picard-definition-for-FBSDE}. Moreover, this Picard
operator has a unique fixed point and is a strict contraction under
the norm
\begin{equation}
\|(Y,Z)\|_\beta := \left\{ \mathbb{E} \int_0^Te^{2\beta
s}\Bigl(|Y_s|^2 + |Z_s|^2\Bigr) \,ds \right\}^{1/2}
\end{equation}
for some constant \(\beta \in \mathbb{R}\).
\end{lemma}

\begin{proof}
See \cite[pp. 140--141]{phamContinuoustimeStochasticControl2009}.
\end{proof}


\begin{lemma}[FBSDE Existence]
\label{lemma:FBSDE-Existence}
Let Assumption~\ref{assumption:standing-assumption-for-FBSDEs}
hold. Then, a pair \((Y, Z) \in \mathcal{M}[0,T]\) is an adapted
solution of FBSDE~\eqref{eq:FBSDE} if and only if it is a fixed point
of the Picard operator defined by
\eqref{eq:Picard-definition-for-FBSDE}.

Furthermore, assume that there exist constants \(L_0\) and \(L_g\)
such that the following inequalities hold almost surely:
$$
\left\{
\begin{aligned}
&|\sigma(t,x,y,\hat{z}) - \sigma(t,x,y,\check{z})|
\leq L_0|\hat{z} - \check{z}|,\\
&\qquad \forall (x,y)\in\mathbb{R}^n\times\mathbb{R}^m,\;
\hat{z},\check{z}\in\mathbb{R}^{m\times d}, \text{ a.e. } t \geq 0,\\
&|g(\hat{x}) - g(\check{x})| \leq L_g|\hat{x} - \check{x}|,\\
&\qquad \forall \hat{x},\check{x}\in\mathbb{R}^n.
\end{aligned}
\right.
$$
If \(L_0L_g < 1\), then there exists a constant \(T_0 > 0\) such that
for any \(T\in(0,T_0]\) and any initial point
\(x_0\in\mathbb{R}^n\), the Picard operator has a unique fixed point
and is a strict contraction under the following norm
\begin{equation}
\|(Y,Z)\|_\textrm{sup} :=
\sup_{t\in[0,T]} \left\{ \mathbb{E} |Y_t|^2 +
\mathbb{E}\int_t^T |Z_s|^2\,ds \right\}^{1/2}.
\end{equation}
\end{lemma}

\begin{proof}
See \cite[pp. 19--22]{maForwardbackwardStochasticDifferential2007}.
\end{proof}

This theorem requires additional assumptions, particularly the
smallness of the time horizon, to establish the contraction property
of \(\Phi\). This is one drawback of applying fixed-point arguments to
prove existence for coupled FBSDEs. There have been extensive studies
aimed at overcoming this limitation, including the monotonicity
condition
\cite{huSolutionForwardbackwardStochastic1995,pengFullyCoupledForwardbackward1999}
and the four-step scheme \cite{maSolvingForwardbackwardStochastic1994}.
Nevertheless, these approaches are compatible with our framework, and
the Picard operator \(\Phi\) remains well-defined under
Assumption~\ref{assumption:standing-assumption-for-FBSDEs}.

To be clear, ``compatible with our framework'' means only that the BML
definition and its Picard interpretation remain meaningful whenever
such alternative theories guarantee well-posedness of the FBSDE and of
the induced equations associated with a trial pair
\((\tilde y,\tilde z)\). It does not mean that these structural
conditions are checked numerically during training. Rather, they serve
as theoretical assumptions under which the objective admits the same
mathematical interpretation.


\subsection{Proof of Lemma~\ref{lemma:mu-norm}}
\label{appendix:norms-on-the-solution-space}




\begin{proof}
First, we simplify the definition of \(\mu\)-norm by interchanging the
order of integrations. For any \((Y, Z) \in
\mathcal{M}[0,T]\),
\begin{equation}
\begin{aligned}
 & \|(Y, Z)\|_\mu \\
  &= \left\{ \mathbb{E} \int_0^T
  |Y_t|^2 \,\mu(dt) + \mathbb{E} \int_{\{0 \leq t \leq T; t \leq s
    \leq T\}} |Z_s|^2 \,ds \, \mu(dt) \right\}^{1/2} \\
  &= \left\{ \mathbb{E} \int_0^T |Y_t|^2 \,\mu(dt) + \mathbb{E}
  \int_0^T \mu([0,s]) |Z_s|^2 \,ds \right\}^{1/2} \\
  &= \left\{
  \mathbb{E} \int_0^T \biggl( |Y_t|^2 + t|Z_t|^2 \biggr)\,
  dt\right\}^{1/2}.
\end{aligned}
\end{equation}
The last equality holds as \(\mu\) is fixed to the Lebesgue measure.

The triangle inequality of \(\|\cdot\|_\mu\) on the solution space can be obtained
from the Minkowski inequality.


Note that \(\|(Y,Z)\|=0\) implies \(\|(Y,Z)\|_\mu=0\). The
converse implication holds too as \((Y, Z) \in \mathcal{M}[0,T]\), which guarantees
that \(Y\) is a continuous process. In particular, there exists a
\(\mathbb{P}\)-null set \(N\) such that for all \(\omega\in N^c\), the sample path
\(Y(\cdot,\omega)\) is continuous. Let \(X(\omega) = \int_0^T|Y(t,\omega)|^2\,dt \geq 0\). By
assumption, \(\mathbb{E}X=0\). Then, there exists a \(\mathbb{P}\)-null set \(N_0\) such
that for all \(\omega \in N_0^c, ~ X(\omega)=0\). Let \(N_1 = N \cup N_0\). Then, for any
\(\omega \in N_1^c\), the sample path \(|Y(\cdot,\omega)|^2\) is continuous and equals
zero almost everywhere on \([0,T]\), implying that \(y(\cdot,\omega)\) is zero
everywhere on \([0,T]\). In conclusion, there exists a \(P\)-null set
\(N_1\) such that \(\sup_{t\in[0,T]}|Y(t,\omega)|^2=0\) for any \(\omega\in N_1^c\),
implying that \(\mathbb{E}\sup_{t\in[0,T]}|Y(t,\omega)|^2=0\).

Finally, the relationships between \(\mu\)-norm~\eqref{eq:mu-norm},
\(\beta\)-norm~\eqref{eq:beta-norm}, sup-norm~\eqref{eq:sup-norm} and
the standard norm~\eqref{eq:standard-norm} are proved below.

\begin{enumerate}
\item The \(\mu\)-norm is weaker than the \(\beta\)-norm. Noting that all \(\beta\)-norms
are equivalent among different \(\beta \in \mathbb{R}\) (as the exponential function is bounded on $[0,T]$),
it suffices to prove
\(\mu\)-norm is weaker than \(\beta\)-norm for \(\beta=1/2\). Indeed,
in this case
$$ \begin{aligned} \|(Y,Z)\|_\beta &\geq \left\{ \mathbb{E}
   \int_0^T|Y_t|^2\,dt + \mathbb{E} \int_0^Tt |Z_t|^2\,dt \right\}^{1/2} . \end{aligned} $$

\item The \(\beta\)-norm is weaker than the sup-norm. Again, it suffices to
discuss the case of \(\beta = 0\). In that case, $$ \begin{aligned}
   \|(Y,Z)\|_\beta &\leq \left\{ T\|(Y, Z)\|^2_\text{sup} + \|(Y,
   Z)\|^2_\text{sup} \right\}^{1/2} \\ &= \sqrt{T+1}
   \|(Y,Z)\|_\text{sup}. \end{aligned} $$

\item The sup-norm is weaker than the standard norm. Noting that
\(\mathbb{E}|Y_t|^2 \leq \mathbb{E} \sup_{t\in[0,T]}|Y_t|^2\) holds for any \(t \in
   [0,T]\). Thus,
   $$ \begin{aligned}
   \|(Y,Z)\|_\text{sup}  &\leq \left\{
   \sup_{t\in[0,T]} \mathbb{E} |Y_t|^2 + \sup_{t\in[0,T]} \mathbb{E}\int_t^T |Z_s|^2\,ds
   \right\}^{1/2} \\
    &\leq \left\{ \mathbb{E} \sup_{t\in[0,T]} |Y_t|^2 +
   \mathbb{E}\int_0^T |Z_s|^2\,ds \right\}^{1/2}. \end{aligned}
   $$
\end{enumerate}
To conclude, the \(\mu\)-norm is the weakest one.
\end{proof}

\subsection{Proof of Lemma~\ref{lemma:contraction-under-a-norm-equivalent-to-mu-norm}}
\label{appendix:contraction-property-of-Picard-operator-for-BSDEs}

We follow a routine approach to analyze the Lipschitz constant of \(\Phi\)
under certain variants of \(\mu\)-norm for BSDE~\eqref{eq:BSDE}.

Let \((\tilde{y},\tilde{z}),(\bar{y}, \bar{z})\in \mathcal{M}[0,T]\). Let
$$ \begin{aligned} (\widetilde{Y},\widetilde{Z}) &:=
\Phi(\tilde{y},\tilde{z}), &\quad (\overline{Y}, \overline{Z}) &:=
\Phi(\bar{y},\bar{z}), \\ \widehat{Y} &:= \widetilde{Y} - \overline{Y},
&\quad \widehat{Z} &:= \widetilde{Z} - \overline{Z}, \\ \hat{y} &:=
\tilde{y} - \bar{y}, &\quad \hat{z}&:= \tilde{z} - \bar{z},
\end{aligned} $$ and \(\hat{f}_t := f(t,\tilde{y}_t,\tilde{z}_t) -
f(t,\bar{y}_t,\bar{z}_t)\). Then, \(\widehat{Y}\) satisfies the following
SDE $$ d\widehat{Y}_t = -\hat{f}_t\,dt + \widehat{Z}_t\,dW_t. $$
Let
\(\beta \in \mathbb{R}\) to be chosen later. Applying Itô's formula to \(te^{2\beta
t}|\widehat{Y}_t|^2\) yields
$$ \begin{aligned}
    d(te^{2\beta t}|\widehat{Y}_t|^2) =&
\bigl[(1+2 \beta t)e^{2\beta t}|\widehat{Y}_t|^2 - 2t e^{2\beta t}\langle \hat{f}_t,
\widehat{Y}_t\rangle \\
&+ te^{2\beta t} |\widehat{Z}_t|^2\bigr]\,dt + 2te^{2\beta t} \langle
\widehat{Y}_t, \widehat{Z}_t\,dW_t\rangle.
\end{aligned} $$
Noting that \(\widehat{Y}_T=0\),
we have
\begin{equation}
\begin{aligned}
\label{eq:tmp24}
0 =& \int_0^T(1+2 \beta t)e^{2\beta t}|\widehat{Y}_t|^2\,dt-
\int_0^T2te^{2\beta t}\langle \hat{f}_t,\widehat{Y}_t\rangle\,dt\\
&+ \int_0^Tte^{2\beta t}|\widehat{Z}_t|^2\,dt +\int_0^T 2 te^{2\beta t}\langle\widehat{Y}_t, \widehat{Z}_t\,dW_t\rangle.
\end{aligned}
\end{equation}

Observe that the stochastic integral vanishes after taking expectation
as the local martingale \(M:=\{\int_0^t se^{2\beta s}\langle \widehat{Y}_s,
\widehat{Z}_s\,dW_s\rangle\}_{0 \leq t \leq T}\) is actually a martingale. To
verify this fact, it suffices to check that \(\sup_{t\in[0,T]}|M_t|\) is
integrable \cite[p. 8]{phamContinuoustimeStochasticControl2009}.
The standard argument using the Burkholder-Davis-Gundy inequality implies that
$\mathbb{E}\Bigl[\sup_{t\in[0,T]} |M_t|\Bigr] < \infty$.

After taking expectation on both sides of \eqref{eq:tmp24}, the stochastic
integral vanishes and $$ \begin{aligned}
& \mathbb{E} \int_0^T(1+2 \beta t)e^{2\beta
t}|\widehat{Y}_t|^2\,dt+\mathbb{E}\int_0^Tte^{2\beta t}|\widehat{Z}_t|^2\,dt \\
&\leq 2\mathbb{E}\int_0^Tte^{2\beta
t}|\hat{f}_t|\cdot |\widehat{Y}_t|\,dt \\
&\leq 2L_f\mathbb{E}\int_0^Tte^{2\beta t}
(|\hat{y}_t| + |\hat{z}_t|) \cdot |\widehat{Y}_t|\,dt \\
&\leq 2L_f\mathbb{E}\int_0^T
te^{2\beta t}\biggl(C_y|\widehat{Y}_t|^2 + \frac{|\hat{y}_t|^2}{4C_y} +
C_z|\widehat{Y}_t|^2 + \frac{|\hat{z}_t|^2}{4C_z}\biggr)\,dt
\\
&=2L_f(C_y+C_z) \mathbb{E}\int_0^T te^{2\beta t}|\widehat{Y}_t|^2\,dt + \frac{1}{2}
\mathbb{E}\int_0^T \frac{tL_f}{C_y} e^{2\beta t}|\hat{y}_t|^2\,dt\\
&\quad+ \frac{1}{2} \mathbb{E}\int_0^T
\frac{L_f}{C_z} te^{2\beta t}|\hat{z}_t|^2\,dt. \end{aligned} $$ Here,
\(C_y\) and \(C_z\) are arbitrary positive constants. Set \(\beta =
L_f(C_y+C_z)\) with \(C_z=L_f\) and \(C_y=TL_f\). Then, $$ \mathbb{E} \int_0^Te^{2\beta
t}\Bigl(|\widehat{Y}_t|^2+t|\widehat{Z}_t|^2\Bigr)\,dt \leq \frac{1}{2}\mathbb{E}
\int_0^Te^{2\beta t}\Bigl(|\hat{y}_t|^2+t|\hat{z}_t|^2\Bigr)\,dt. $$ This
suggests that \(\Phi\) is a strict contraction under a certain variant of
the \(\mu\)-norm: there exists a norm equivalent to \(\mu\)-norm, denoted
by \(\|\cdot\|_{\mu(\beta)}\), such that for any \((\tilde{y},\tilde{z}),(\bar{y},
\bar{z})\in \mathcal{M}[0,T]\), $$ \|\Phi(\tilde{y},\tilde{z}) - \Phi(\bar{y},
\bar{z})\|_{\mu(\beta)} \leq \frac{1}{2} \|(\tilde{y},\tilde{z})-(\bar{y},
\bar{z})\|_{\mu(\beta)}. $$
\subsection{More details of Numerical Examples}
\label{appendix:more-details-of-numerical-examples}
We provide more details of numerical examples omitted in the main body.

\textbf{More details of Example~\ref{example:a-toy-BSDE-revision-1} and Example~\ref{example:a-coupled-FBSDE-revision-1}}.
A reasonably good Monte Carlo estimation with small confidence-intervals may require enough samples. In Figure~\ref{fig:visualize-BML-toy-BSDE}, the empirical BML value for each pair $(\theta_1,\theta_2)$ is estimated using $10^6$ Monte Carlo samples. Below, we reproduce the same figure using only $10^3$ Monte Carlo samples; results are presented in Figure~\ref{fig:visualize-BML-toy-BSDE-1000-samples}.

Similarly, we reproduce Figure~\ref{fig:visualize-BML-coupled-FBSDE} using only $10^3$ Monte Carlo samples; results are presented in Figure~\ref{fig:visualize-BML-coupled-FBSDE-1000-samples}.

\textbf{More details of Example~\ref{example:a-toy-BSDE-revision-2}}. The analytical expression~\eqref{eq:theoretical-BML-toy-BSDE-parameterization-scheme-2} is calculated via Theorem~\ref{theorem:the-Picard-interpretations-of-BML-values}. Note that BSDE~\eqref{eq:a-toy-BSDE} is simple enough such that $\Phi(\tilde{y},\tilde{z})$ is exactly the true solution $(Y^*,Z^*)$ for any trial solution $(\tilde{y},\tilde{z})$. Therefore,
\begin{equation}
\label{eq:appendix-tmp-26}
\begin{aligned}
\operatorname{BML}(\theta_1,\theta_2)
=&
\| (\tilde{y}, \tilde{z}) - \Phi(\tilde{y}, \tilde{z}) \|_\mu^2
= \| (\tilde{y}, \tilde{z}) - (Y^*,Z^*) \|_\mu^2 \\
=& \mathbb{E}\int_0^T|W_t|^4\Bigl(\theta_1|W_t|^2- \frac{1}{d}
\Bigr)^2\,dt\\
&+ \mathbb{E}\int_0^T t|W_t|^2\Bigl(\theta_2|W_t|^2-
\frac{2}{d}\Bigr)^2\,dt \\
=& \ell_y^*(\theta_1;\|\cdot\|_\mu) + \ell_z^*(\theta_2;\|\cdot\|_\mu),
\end{aligned}
\end{equation}
where $\ell^*_y(\theta_2;\|\cdot\|_\mu)$ and
$\ell^*_z(\theta_2;\|\cdot\|_\mu)$ are quadratic functions given by
(noting $\mathbb{E}|W_t|^{2k} = d(d+2)\cdots(d+2k-2)
t^k$).
A direct calculation shows that these quadratic functions achieve
minimum at $\theta_1^*= \frac{5}{4d(d+6)T}$ and $\theta_2^*=
\frac{5}{2d(d+4)T}$.

We note that the empirical BML line in Figure~\ref{fig:optimize-BML-toy-BSDE} should align with this analytical expression, and have exactly the same minimum. However, as we use only $10^3$ Monte Carlo samples at each gradient step when estimating empirical BML, the obtained empirical line contains too much noise (c.f. Figure~\ref{fig:visualize-BML-toy-BSDE} and Figure~\ref{fig:visualize-BML-toy-BSDE-1000-samples}). In fact, the exact BML value decreases smoothly even using inaccurate gradient estimations. We recreate Figure~\ref{fig:optimize-BML-toy-BSDE} with the right panel showing exact BML values obtained from the analytical expression \eqref{eq:appendix-tmp-26}; results are presented in Figure~\ref{fig:optimize-BML-toy-BSDE-theoretical-BML}.

\textbf{More details of Example~\ref{example:a-coupled-FBSDE-revision-2}}.
While Figure~\ref{fig:optimize-BML-coupled-FBSDE} shows the evolution
of norms during the optimization process, we can also visualize the
mean square errors $\mathbb{E}|\tilde{y}_t - Y^*_t|^2$ and
$\mathbb{E}|\tilde{z}_t - Z^*_t|^2$ as functions of time $t$ at a particular
optimization step. These error ``paths'' provide more information into
the performance of $(\tilde{y},\tilde{z})$ than aggregate errors
calculated from norms.

The sample paths of the trial solution $(\tilde{y},\tilde{z})$ are
constructed as follows. First, we run the optimization process
described in Example~\ref{example:a-coupled-FBSDE-revision-2},
yielding neural networks $(\tilde{y}^{\theta}, \tilde{z}^{\theta})$.
This optimization is repeated independently $50$ times, thus producing
a collection $\{(\tilde{y}^{\theta_k},
\tilde{z}^{\theta_k})\}_{k=1}^{50}$. Next, $1000$ independent sample
paths of the driving Brownian motion $W$, denoted by
$\{W^{(j)}\}_{j=1}^{1000}$, are generated. For each path $W^{(j)}$,
and for each neural network parameter set $\theta_k$, we simulate the
corresponding trial solutions $(\tilde{y}^{(j),k}, \tilde{z}^{(j),k})$
according to the parameterization scheme specified in
Example~\ref{example:a-coupled-FBSDE-revision-2}. Finally, we average
the trial solutions over $k$ and regard $\{(\tilde{y}^{(j)},
\tilde{z}^{(j)})\}_{j=1}^{1000}$ as the final sample paths of
$(\tilde{y},\tilde{z})$.

\begin{figure}[h]
  \centering
  \includegraphics[width=.45\textwidth]{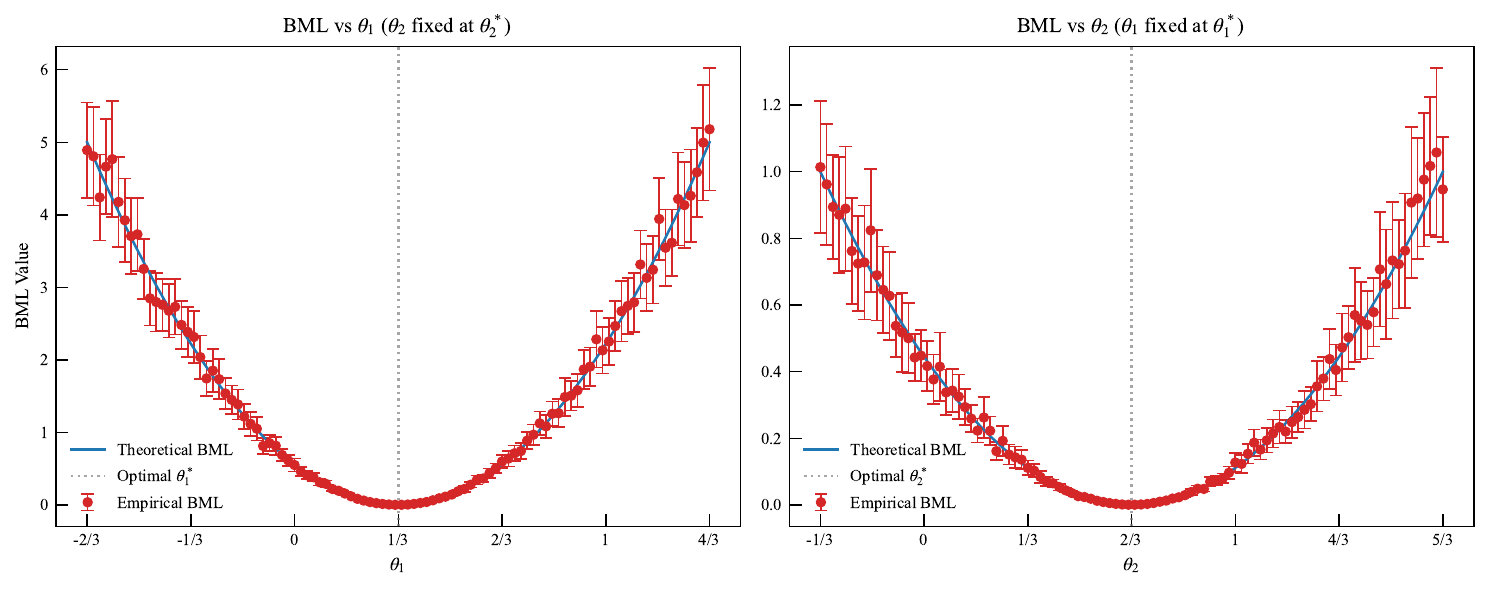}
  \caption{\label{fig:visualize-BML-toy-BSDE-1000-samples}Reproduce Figure~\ref{fig:visualize-BML-toy-BSDE} using only $10^3$ Monte Carlo samples when estimating empirical BML.
  }
\end{figure}

\begin{figure}[h]
  \centering
  \includegraphics[width=.45\textwidth]{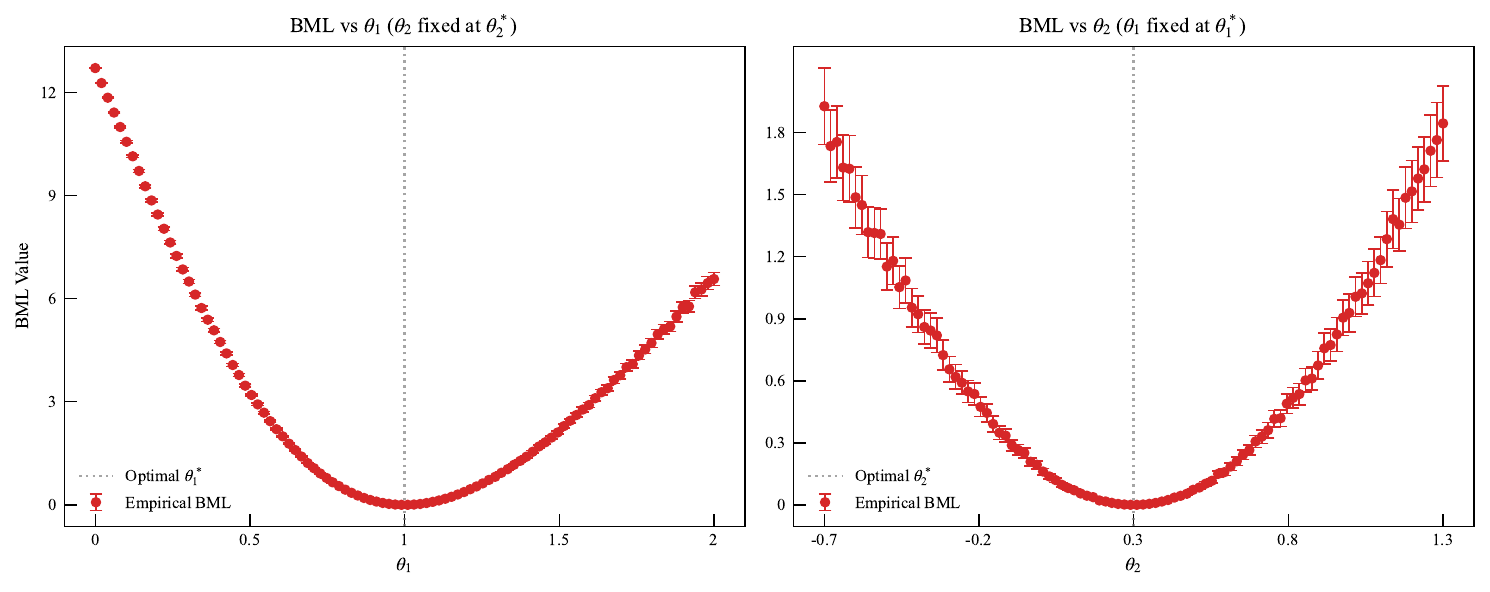}
  \caption{\label{fig:visualize-BML-coupled-FBSDE-1000-samples}Reproduce Figure~\ref{fig:visualize-BML-coupled-FBSDE} using only $10^3$ Monte Carlo samples when estimating empirical BML.
  }
\end{figure}

\begin{figure}[h]
  \centering
  \includegraphics[width=.45\textwidth]{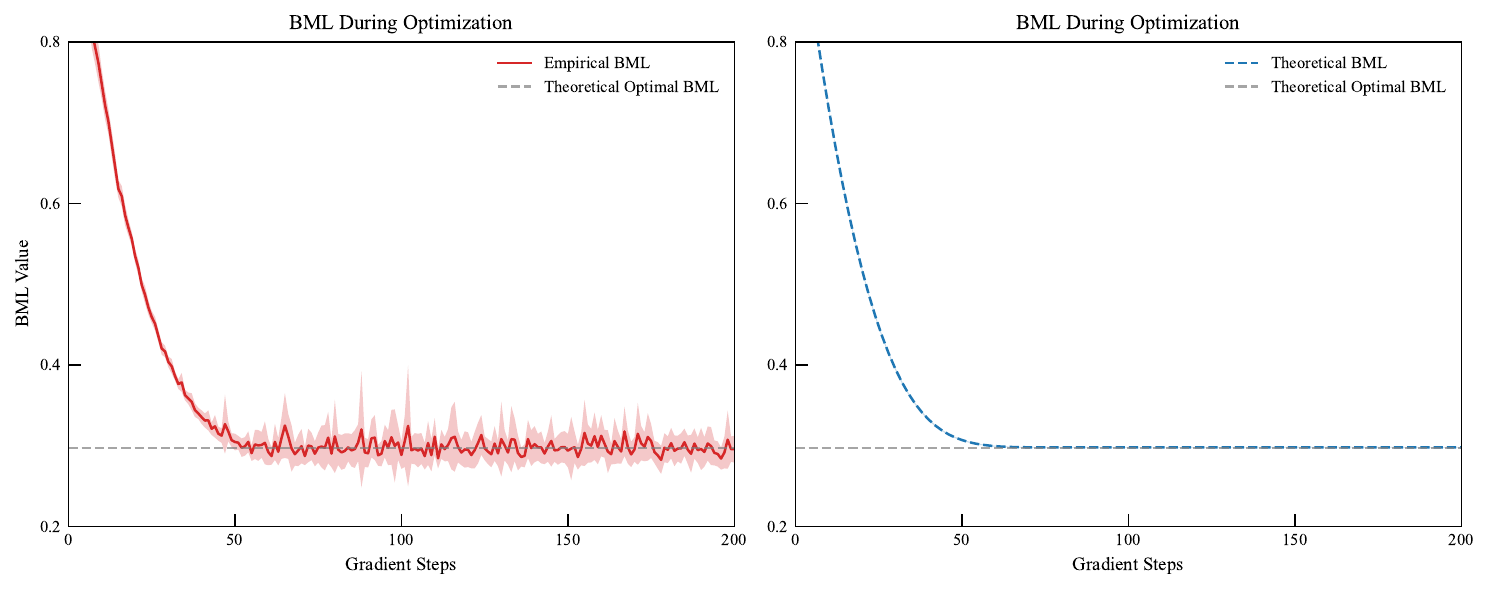}
  \caption{\label{fig:optimize-BML-toy-BSDE-theoretical-BML}Reproduce
    Figure~\ref{fig:optimize-BML-toy-BSDE} with the right panel
    showing theoretical BML values during the optimization. Metrics
    averaged over 50 independent runs; shaded regions indicate $\pm 3$
    standard errors. The shaded region of the theoretical line is
    nearly invisible due to small standard errors.}
\end{figure}

It is important to note that the parameterization scheme in
Example~\ref{example:a-coupled-FBSDE-revision-2} does not specify the
time discretization used during training or evaluation. This
flexibility allows the neural networks to be trained with a relatively
coarse temporal grid while the resulting trial solutions
$(\tilde{y},\tilde{z})$ are evaluated and visualized with a much finer
temporal resolution. In this example, the neural networks are
trained using $20$ time intervals, whereas the error paths are
visualized on a grid with $1000$ time intervals; results are presented
in Figure~\ref{fig:error-path-coupled-FBSDE}.

\textbf{More details of Example~\ref{example:a-HJB-decoupled-FBSDE}}. We reproduce Figure~\ref{fig:optimize-BML-HJB-decoupled-FBSDE} and Table~\ref{table:Y0-error-HJB-decoupled-FBSDE} after training a total of 4000 gradient steps; results are presented in Figure~\ref{fig:optimize-BML-HJB-decoupled-FBSDE-4k} and Table~\ref{table:Y0-error-HJB-decoupled-FBSDE-4k}. Again, we observe that smaller BML values correspond to smaller relative errors.

\begin{figure}[h]
  \centering
  \includegraphics[width=.45\textwidth]{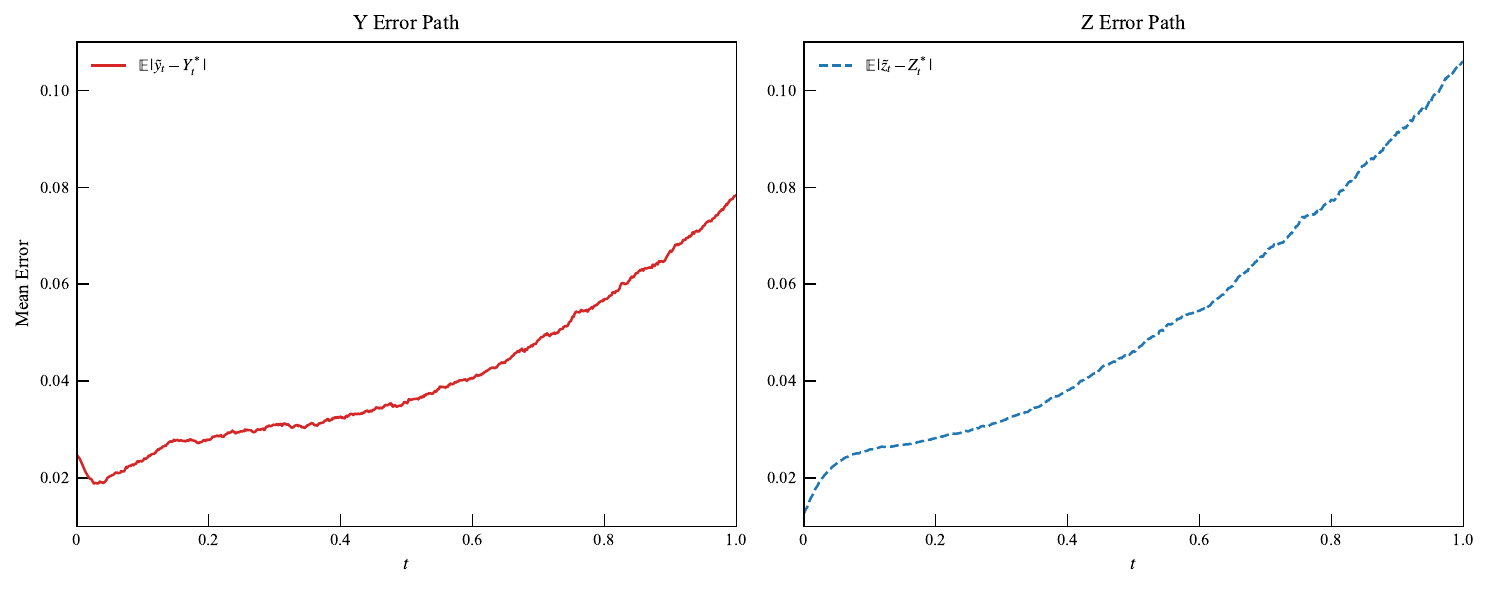}
  \caption{\label{fig:error-path-coupled-FBSDE} Error paths for
    Example~\ref{example:a-coupled-FBSDE-revision-2} between the trial
    solution $(\tilde{y},\tilde{z})$ and the true solution
    $(Y^*,Z^*)$. Left: mean squared $\mathbb{E}|\tilde{y}_t -
    Y^*_t|^2$ for $Y$. Right: mean squared error
    $\mathbb{E}|\tilde{z}_t - Z^*_t|^2$ for $Z$.}
\end{figure}

\begin{figure}[h]
  \centering
  \includegraphics[width=.45\textwidth]{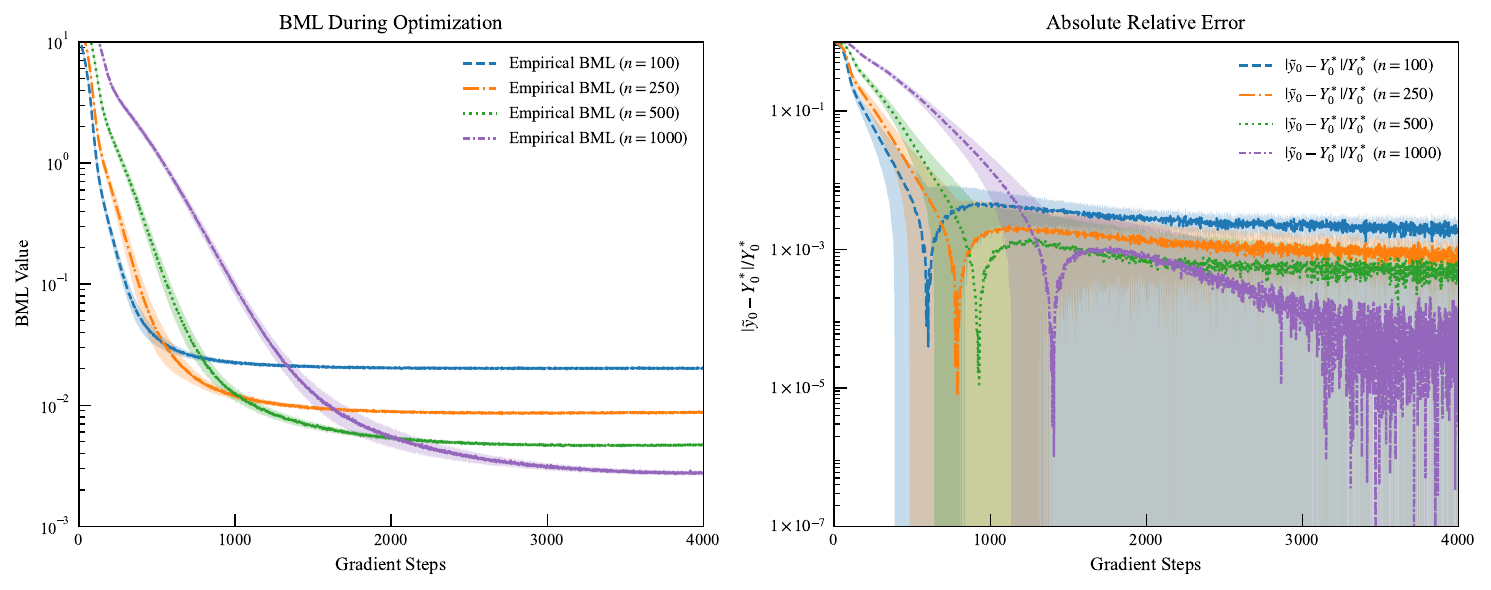}
  \caption{\label{fig:optimize-BML-HJB-decoupled-FBSDE-4k}Reproduce
    Figure~\ref{fig:optimize-BML-HJB-decoupled-FBSDE} after training a
    total of 4000 gradient steps. Left: empirical BML values. Right:
    relative errors $|\tilde{y}_0-Y_0^*|/Y_0^*$. Metrics averaged over
    50 independent runs; shaded regions indicate $\pm 3$ standard
    errors.}
\end{figure}

\begin{table}[h]
  \centering
\caption{Reproduce Table~\ref{table:Y0-error-HJB-decoupled-FBSDE}
  after training a total of 4000 gradient steps.}
\label{table:Y0-error-HJB-decoupled-FBSDE-4k}
\begin{tabular}{ccccc}
\toprule
Dimension & Final BML & Optimal cost & Prediction & Rel. error \\
($n$) & value & ($Y_0^*$) & ($\tilde{y}_0$) & ($|\tilde{y}_0 - Y_0^*|/Y_0^*$) \\
\midrule
100 & 0.02027 & 4.59016 & 4.59831 & 0.178\% \\
250 & 0.00871 & 5.51545 & 5.52012 & 0.084\% \\
500 & 0.00471 & 6.21161 & 6.21520 & 0.058\% \\
1000 & 0.00271 & 6.90626 & 6.90623 & 0.0004\% \\
\bottomrule
\end{tabular}
\end{table}

\section*{References}
\bibliographystyle{IEEEtran}
\bibliography{solveFBSDE}

\end{document}